\documentclass[11pt, a4paper]{amsart}

\usepackage{amsfonts,amsthm,amssymb,amsmath}
\numberwithin{equation}{section}
\usepackage{latexsym}
\usepackage[dvips]{graphicx}
\usepackage{geometry}
\usepackage{hyperref}

\title {Lyapunov exponents in Hilbert geometry\\}

\author{Micka\"el Crampon}
\email{crampon@math.unistra.fr}

\newtheorem{defi}{Definition}[section]
\newtheorem{defis}[defi]{Definitions}
\newtheorem{thm}[defi]{Theorem}

\newtheorem{prop}[defi]{Proposition}
\newtheorem{lemma}[defi]{Lemma}
\newtheorem{corollary}[defi]{Corollary}

\newtheorem{rmk}[defi]{Remark}
\newtheorem{conjecture}[defi]{Conjecture}

\def\C{\mathcal{C}}

\def\P{\mathbb{P}}

\def\R{\mathbb{R}}

\def\ph{\varphi}

\def\doo{\partial\Omega}

\def\o{\Omega}
\def\g{\gamma}
\def\l{\lambda}

\def\H{\mathcal{H}}

\def\span{\textrm{span}}

\def\G{\Gamma}

\def\d{d_{\o}}

\begin{document}

\maketitle

\begin{abstract}
We study the behaviour of a Hilbert geometry when going to infinity along a geodesic line. We prove that all the information is contained in the shape of the boundary at the endpoint of this geodesic line and have to introduce a regularity property of convex functions to make this link precise.\\
The point of view is a dynamical one and the main interest of this article is in Lyapunov exponents of the geodesic flow.
\end{abstract}
\tableofcontents
\newpage

\section{Introduction}

This article is meant to be a contribution to the understanding of Hilbert geometries, by a study of their behaviour when approaching infinity. Most of this work is part of my Ph.D. thesis, which can be found in various places on the Internet.\\

\subsection{Context}

A Hilbert geometry is a metric space $(\o,\d)$ where
\begin{itemize}
 \item $\o$ is a \emph{proper open convex set} of the real projective space $\R\P^n$, $n\geqslant2$; \emph{proper} means there exists a projective hyperplane which does not intersect the closure of $\o$, or, equivalently, there is an affine chart in which $\o$ appears as a relatively compact set;
 \item $\d$ is the distance on $\o$ defined, for two distinct points $x,y$, by
$$\d(x,y) =\frac{1}{2}|\log[a,b,x,y]|,$$
where $a$ and $b$ are the intersection points of the line $(xy)$ with the boundary $\partial \Omega$ and $[a,b,x,y]$ denotes the cross ratio of the four points : if we identify the line $(xy)$ with $\R\cup\{\infty\}$, it is defined by $[a,b,x,y]=\frac{|ax|/|bx|}{|ay|/|by|}$ .
\end{itemize}

\begin{figure}[!h]
\begin{center}
\includegraphics{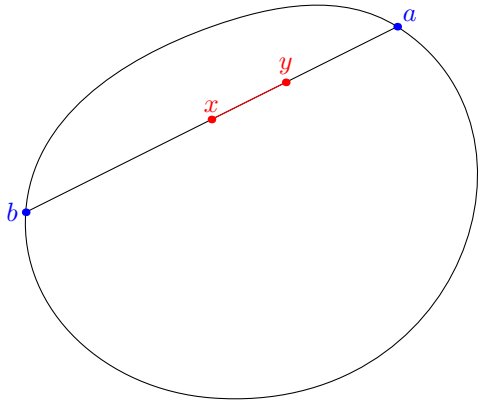}
\end{center}
\caption{The Hilbert distance}
\label{fighilbert}
\end{figure}

These geometries had been introduced by Hilbert at the end of the nineteenth century as examples of spaces where lines would be geodesics, which one can see as a motivation for the fourth of his famous problems, which roughly consisted in finding all geometries satisfying this property.\\

Different Hilbert geometries can have very different geometric behaviours. For example, the geometry defined by a triangle in $\R\P^2$ is isometric to the $2$-dimensional real space equipped with a norm whose ball is a regular hexagon \cite{delaharpe}; on the other side, the geometry defined by an ellipsoid is precisely the model that Beltrami proposed for hyperbolic geometry.\\
Classifying Hilbert geometries happens to be a quite difficult task, but the global feeling is that any Hilbert geometry has an intermediate behaviour in between Euclidean and hyperbolic geometry. Most of the previous works attempted to determinate those Hilbert geometries which resembles more Euclidean or hyperbolic space:

\begin{itemize}
 \item \cite{foertschkarlsson} $(\o,\d)$ is isometric to a normed vector space if and only if $\o$ is a simplex.
 \item \cite{cvv} \cite{vernicos} \cite{bernig} The following statements are equivalent:
\begin{itemize}
\item $\o$ is a polytope;
\item $(\o,\d)$ is bi-Lipschitz equivalent to the Euclidean space;
\item $(\o,\d)$ is quasi-isometric to the Euclidean space.
\end{itemize}
 \item \cite{cv} If $\o$ is strongly convex, that is, $\doo$ is of class $\C^2$ with positive definite Hessian, then $(\o,\d)$ is bi-Lipschitz equivalent to the hyperbolic space.
\end{itemize}

These results only consider polytopes or strongly convex sets and, as soon as we permit more irregularity or less symmetry, no global behaviour can be expected. Here we should recall the works of Yves Benoist who studied less regular Hilbert geometries, in particular those which admit compact quotients, called \emph{divisible convex sets}. For the problem of classification we are concerned with here, the major achievement of Benoist is probably the characterization of Gromov-hyperbolic Hilbert geometries: they are those defined by quasi-symmetrically convex sets \cite{benoistqs}. About divisible convex sets, Benoist proved an hyperbolic/non hyperbolic alternative in \cite{benoistcv1}: if $\o$ is a divisible convex set, then the following are equivalent:

\begin{itemize}
 \item $\o$ is strictly convex;
 \item $\doo$ is of class $\C^1$;
 \item $(\o,\d)$ is Gromov-hyperbolic.\\
\end{itemize}

The goal of the present work is to get interested in all those forgotten Hilbert geometries which enjoy neither high regularity nor numerous symmetries, the strategy being the following: pick a geodesic ray (a line), follow this line to infinity and look at the geometry around it.\\
Consider as an easy example the Hilbert geometry defined by a half disc, and call $a$ and $b$ the extremities of the diameter. Pick two distinct geodesic rays $c_1, c_2 : [0,+\infty) \longrightarrow \o $, ending at points $x_1$ and $x_2$ in $\doo$. 
\begin{itemize}
 \item Assume $x_1\not = x_2$. The distance between the two geodesic rays goes to infinity, except when both points $x_1$ and $x_2$ are inside the open segment $]ab[$, in which case one can parametrize the rays such that
$$\lim_{t\to +\infty} \d(c_1(t),c_2(t)) = \frac{1}{2} |\log [a b x_1 x_2]|.$$
 \item Assume $x_1=x_2=:x$. If $x=a$ (or $x=b$), the distance between the two geodesic rays tends to some positive constant $d>0$, whose value depends on the parametrization; the smallest of which being $d = \frac{1}{2} |\log [(ab) D c_1 c_2]|$, where $D$ is the line tangent to the half-circle at $a$ (or $b$), and $[(ab) D c_1 c_2]$ denotes the cross-ratio of the four lines.\\
In the other cases, one can parametrize the rays such that the distance $\d(c_1(t),c_2(t))$ decreases to $0$. Nevertheless, it does not go to $0$ at the same rate: if $x$ is in the (open) circular part, then $\d(c_1(t),c_2(t)) \sim e^{-t}$; if $x$ is in the flat part $]ab[$ then $\d(c_1(t),c_2(t)) \sim e^{-2t}$.
\end{itemize}
These simple remarks show, first, that the boundary at infinity given by asymptotic geodesic rays does not correspond to the geometric boundary $\doo$ and, second, that the geometry when going to infinity heavily depends on the point we are aiming at. This work studies this second point in details.\\
For what concerns the first one, notice that geodesic and geometric boundaries will correspond if and only if the convex set is strictly convex and has $\C^1$ boundary. For polytopes or even more general non-strictly convex sets, another problem arises: there can be geodesics which are not lines. In these cases, the best thing is probably to look at the Busemann boundary, as made in \cite{walsh}, which contains the geometric and geodesic boundaries.\\\\

\subsection{What we study here}

In this article, we focus on those Hilbert geometries defined by a strictly convex set with $\C^1$ boundary. Since our aim is to look at the geometry around a specific geodesic line going to a point $x\in\doo$, we could equivalently assume that $x$ is an extremal point of $\o$ and that $\doo$ is $\C^1$ at $x$. This assumption is then not a very restrictive one, and we can illustrate most of the interesting behaviours; furthermore, it allows us to use and make connections with some differential and dynamical objects that I already used in a previous work \cite{crampon}. In section \ref{extension}, we explain how to get rid of this restriction and extend the main achievements.\\

So we want to understand how the distance $\d(c_1(t),c_2(t))$ between two well parametrized asymptotic geodesic rays decreases to $0$ when $t$ goes to infinity. In particular, as suggested by the example of the half-disc, we would like to see when the decreasing is exponential, and in this case, to determinate the exponential rate.\\
In the case of a strongly convex set, it is easy to see, as we already saw in the case of the half-disc, that $\d(c_1(t),c_2(t))\sim e^{-t}$, as in the hyperbolic space. The main result of this article about this is probably corollary \ref{maincorollary2}, that says that all these informations are enclosed in the shape of the boundary at the endpoint.\\ 


I should confess that the original motivation of this work is not of a geometric nature but of a dynamical one. It is inspired by proposition 5.4 of \cite{crampon}, which I wanted to generalize in order to understand Lyapunov exponents, decomposition and manifolds, associated to the geodesic flow of the Hilbert metric. The text is then written in this spirit, and the geodesic flow is the main object that is studied.\\
The geodesic flow is the flow $\ph^t$ defined on the homogeneous tangent bundle $H\o = T\o \smallsetminus \{0\} / \R_+^*$, which consists of pairs $(x,[\xi])$, where $x$ is a point of $\o$ and $[\xi]$ a direction tangent to $\o$ at $x$. To find the image of a point $w=(x,[\xi])\in H\o$ by $\ph^t$, one follows the geodesic line $c_{w}$ leaving $x$ in the direction $[\xi]$, and one has $\ph^t(w) = (c_{w}(t), [c_w'(t)])$.\\
The geodesic flow is generated by the vector field $X: H\o \longrightarrow TH\o$. If we choose an affine chart and a Euclidean norm $|\ .\ |$ on it in which $\o$ appears as a bounded convex set, then $X$ is related to the generator $X^e$ of the Euclidean geodesic flow by $X=m X^e$, where $m:H\o \longrightarrow \R$ is defined by
$$m(x,[\xi]) = \frac{2}{\displaystyle\frac{1}{|xx^+|} + \frac{1}{|xx^-|}},$$
where $x^+$ and $x^-$ are the intersection points of the line $x+\R.\xi$ with the boundary $\doo$. In particular, we see that, under our hypothesis of $\C^1$ regularity of $\doo$, the function $m$ and the geodesic flow itself are of class $\C^1$.\\

\begin{figure}[!h]
\begin{center}
\includegraphics{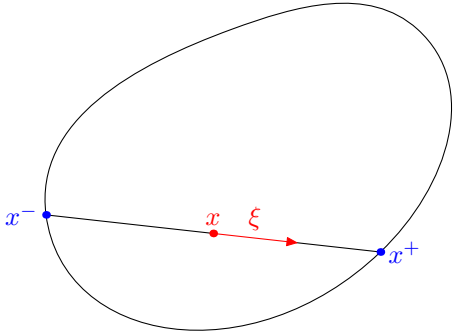}
\end{center}
\caption{The Finsler metric}
\label{figfinsler}
\end{figure}

This fact has to be related with the Finsler nature of the Hilbert metric. Indeed, the Hilbert metric is generated by a field of norms $F: T\o \longrightarrow \R$, with
$$F(x,\xi) = \frac{|\xi|}{2} \left(\frac{1}{|xx^+|} + \frac{1}{|xx^-|}\right) = \frac{|\xi|}{m(x,[\xi])}.$$
By \emph{generated}, we mean that the Hilbert distance between two points $x,y\in\o$ is given by
$$\d(x,y) = \inf_{c : x\to y} \int_{0}^{1} F(\dot c(t))\ dt,$$
where the infimum is taken with respect to all $\C^1$ curves $c: [0,1] \longrightarrow \o$ such that $c(0)=x$, $c(1)=y$.\\\\

\subsection{Contents}
The geodesic flow of Hilbert metrics has been studied by Yves Benoist in \cite{benoistcv1} and by myself in \cite{crampon}. In the second section of this article, I recall the dynamical objects I had used in \cite{crampon} and the fundamental results about them; in particular, the existence of stable and unstable distributions, so that $TH\o$ admits a $\ph^t$-invariant decomposition
$$TH\o = \R.X \oplus E^s \oplus E^u.$$
Stable and unstable distributions are characterized by the fact that, for a stable (resp. unstable) vector $Z\in TH\o$, the norm
$\|d\ph^t(Z)\|$ decreases to $0$ when $t$ goes to $+\infty$ (resp. $-\infty$); the Finsler norm $\|\ .\ \|$ on $H\o$ that we consider here is naturally related to the Finsler metric on $\o$ (see section \ref{metriconHo}).\\
These two distributions are tangent to the stable and unstable foliations $W^s$ and $W^u$ of $H\o$. If one takes a point $w_0 = (x_0,[\xi_0])\in H\o$, its orbit in the future $\{\ph^t(w_0)\}_{t\geqslant 0}$ projects on the geodesic ray $c=\{x_0+\l \xi_0\}_{\l\in\R}$; the orbits, in the future, of the points $w$ in the leaf $W^s(w_0)$ passing through $w_0$ of the stable foliation, project to those geodesic rays $c$ such that $\d(c(t),c_0(t))$ tends to $0$ when $t$ goes to $+\infty$.\\
The goal is then to understand how the norm $\|d\ph^t(Z)\|$ of a stable vector $Z$ goes to $0$ when $t$ goes to $+\infty$; results about distances between geodesic rays will follow by integration.\\

The third and fourth parts look at the exponential growth rate of these norms $\|d\ph^t(Z)\|$, for a stable vector $Z$. This is captured by the following limit, when it exists:
$$\chi(Z) = \lim_{t\to+\infty} \frac{1}{t} \log \|d\ph^t(Z)\|;$$
the quantity $\chi(Z)$ is called the Lyapunov exponent of the vector $Z$. These numbers are investigated in section 3, and section 4 shows that all the information about them is contained in the shape of the boundary at the endpoint of the geodesic ray that had been chosen. This needs the introduction of a new regularity property that we call \emph{approximate regularity}, whose study requires some time in section 4.\\

In section 5, we state the main consequences about the asymptotic behaviour of distances when following a geodesic line and explain how to extend it to the nonregular cases. We also show how Lyapunov submanifolds of the geodesic flow appear very naturally in our context.\\

The sixth part is dedicated to examples, with a focus on divisible convex sets, while the last one gives connections with volume entropy, whose study might benefit from the present work.

\vspace{1cm}

\begin{center}
 
\emph{\Large Unless it is explicitly stated, in particular in sections \ref{extension} and \ref{sectionvolentropy}, the convex set $\o$ is always assumed to be strictly convex with $\C^1$ boundary.}
\end{center}


\vspace{1cm}

\section{Foulon's dynamical formalism and consequences}

\subsection{Dynamical decomposition}

In \cite{crampon}, I explained why the dynamical objects introduced by Patrick Foulon in \cite{foulon86} to study smooth second order differential equations were still relevant and useful in the case of a Hilbert geometry defined by a strictly convex set with $\C^1$ boundary. I briefly recall them here, and refer the reader to \cite{crampon}, \cite{foulon86} or the appendix of \cite{fouloneng}.\\

All the operators, functions or vector fields that we will consider are $C_X$-regular, or equivalently $C_{X^e}$-regular. That means that they are smooth in the direction $X$ of the flow. This is the essential regularity that we need because Hilbert geometries are flat geometries. Remark that this notion makes sense for those objects which are only defined along one specific orbit of the flow.\\
The \emph{vertical bundle} $VH\o$ is the smooth subbundle of vertical vectors, which are tangent to the fibers; it is defined as $VH\o=\ker d\pi$, and has dimension $n-1$. By the letter $Y$, we will always denote a $C_X$-vertical vector field. The \emph{vertical operator} $v_X$ is well defined (this has to be checked) on $TH\o$ by
$$v_X(X)=v_X(Y)=0,\ v_X([X,Y])=-Y,\ Y\in VH\o.$$
The operators $v_X$ and $v_{X^e}$ are related by
$$v_X=mv_X^e.$$

The \emph{horizontal operator} $H_{X}: VH\o \longrightarrow TH\o$ is the $C_X$-linear operator defined by
$$H_{X}(Y) = -[X,Y] - \frac{1}{2} v_{X} ([X,[X,Y]]),\ Y\in VHM.$$
The \emph{horizontal bundle} $h^XH\o$ is the $C_X$-regular subbundle defined as the image of $VHM$ by $H_X$. An important property is the one which relates the operators $H_X$ and $H_{X^e}$:
\begin{equation}\label{horizontal} H_{X}(Y) = mH_{X^e}(Y) + L_Y m X^e + \frac{1}{2}L_{X^e} m Y. \end{equation}

The tangent bundle of $H\o$ admits then a $C_X$-regular decomposition into
$$TH\o = VH\o \oplus h^XH\o \oplus \R.X,$$
which is the counterpart of the Levi-Civita connection for Riemannian metrics.\\

The $C_X$-linear operator $J^X:  VH\o \oplus h^XH\o \longrightarrow  VH\o \oplus h^XH\o$ is defined as $J_X=v_X$ on $h^XH\o$ and $J^X=-H_X$ on $h^XH\o$. It provides a pseudo-complex structure on $VH\o \oplus h^XH\o$: $J^X$ satisfies $J^X \circ J^X=-Id$ and exchanges $VH\o$ and $h^XH\o$.

\subsection{Dynamical derivation and parallel transport}\label{sectionparom}

As an analog of the covariant derivation along $X$, the \emph{dynamical derivation} $D^{X}$ is the $C_X$-differential operator of order 1 defined by
$$D^{X}(X)=0, \ D^{X}(Y)=\displaystyle -\frac{1}{2} v_{X}([X,[ X,Y]]), \ [D^{X}, H_{X}]= 0,\ Y\in VH\o.$$

Being a $C_X$-differential operator of order 1 means that for any function $f\in C_X$,
$$D^{X}(fZ)=fD^{X}(Z) + (L_{X}f) Z.$$
On $VH\o$, we can write
\begin{equation}\label{dxhori} D^{X}(Y) = H_{X}(Y) +[X,Y]. \end{equation}
The operators $D^X$ and $D^{X^e}$ are related by
\begin{equation}\label{formuladerivation} D^{X}=mD^{X^e} + \frac{1}{2} (L_{X^e} m)Id.\end{equation}

A vector field $Z$ is said to be \emph{parallel} along $X$, or along any orbit of the flow if $D^{X}(Z)=0$. This allows us to consider the \emph{parallel transport} of a $C_X$-vector field along an orbit: given $Z(w)\in T_wH\o$, the parallel transport of $Z(w)$ along $\ph.w$ is the parallel vector field $Z$ along the orbit $\ph.w$ of $w$ whose value at $w$ is $Z(w)$; the parallel transport of $Z(w)$ at $\ph^t(w)$ is the vector $Z(\ph^t(w))=T^t(Z(w)) \in T_{\ph^t(w)}H\o$. Since $D^X$ commutes with $J^X$, the parallel transport also commutes with $J^X$. If $X$ is the generator of a Riemannian geodesic flow, the projection on the base of this transport coincides with the usual parallel transport along geodesics.\\

We can relate the parallel transports with respect to $X^e$ and $X$, as stated in the next lemma. This lemma is essential in this work and will be used in many different parts.

\begin{lemma}\label{horver}
Let $w\in HM$ and pick a vertical vector $Y(w)\in V_wHM$. Denote by $Y$ and $Y^e$ its parallel transports with respect to $X$ and $X^e$ along the orbit $\ph.w$. Let $h=J^X(Y)$ and $h^e=J^{X^e}(Y^e)$ be the corresponding parallel transports of $h(w)=J^X(Y(w))$ and $h^e(w)=J^{X^e}(Y^e(w))$ along $\ph.w$. Then
$$Y= \left(\frac{m(w)}{m}\right)^{1/2} Y^e$$
and
$$h= -L_Ym\ X^e +(m(w)m)^{1/2}\ h^e - \frac{m(w)}{m}\ L_{X^e}m\ Y^e.$$
\end{lemma}

\begin{proof}
We look for the unique vector field $Y$ along $\ph.w$ such that $D^{X}(Y)=0$ and which takes the value $Y(w)$ at the point $w$. Equation (\ref{formuladerivation}) gives
$$D^{X}(Y)=mD^{X^e}(Y) + \frac{1}{2} L_{X}(\log m) Y.$$
Assume we can write $Y=fY^e$ along $\ph.w$. Then $f$ is the solution of the equation
$$L_{X} (\log f) + \frac{1}{2} L_{X}(\log m)=0,$$
which, with $f(w)=1$, gives
$$f( \ph^t(w))=\left(\frac{m(w)}{m(\ph^t(w))}\right)^{1/2}.$$
Finally,
\begin{equation}\label{tut}
Y(\ph^t w)= \left(\frac{m(w)}{m(\ph^t(w))}\right)^{1/2} Y^e(\ph^t w).
\end{equation}

Now, using (\ref{dxhori}), we have
$$h=H_{X}(Y)=-[X,Y]+D^{X}(Y)=-[X,Y]$$ along $\ph.w$.
Hence, from (\ref{tut}), we have

$$\begin{array}{rcl}
h=-[ X,Y] & = & -L_Ym\ X^e-m\ [X^e,Y]\\\\
 & = & -L_Ym\ X^e-m\ [X^e, \frac{m(w)}{m}Y^e] \\\\
& =& -L_Ym\ X^e -(m(w)m)^{1/2}\ [X^e,Y^e] + m(w) m\ L_{X^e}(m^{-1})\ Y^e\\\\
&= & -L_Ym\ X^e +(m(w)m)^{1/2}\ h^e  - \frac{m(w)}{m}\ L_{X^e}m\ Y^e.
\end{array}$$

\end{proof}

\subsection{Metrics on $H\o$}\label{metriconHo}

Dynamical flows are usually studied on Riemannian manifolds, and most of the definitions or theorems are stated in this context. In the case of geodesic flows on complete Riemannian manifolds $M$, $HM$ inherits a natural Riemannian metric from the base metric. In our case, we define a Finsler metric $\|\ .\ \|$ on $H\o$, using the decomposition $TH\o = \R.X \oplus h^{X}H\o \oplus VH\o$: if $Z=aX+h+Y$ is some vector of $TH\o$, we set
\begin{equation}\label{metriconHM} \|Z\| = \left(|a|^2+\frac{1}{2}\left((F(d\pi h))^2+(F(d\pi J^{X} (Y)))^2\right)\right)^{1/2}. \end{equation}
Since the last decomposition is only $C_X$-regular in general, $\|\ .\ \|$ is also only $C_X$-regular. It allows us to define the length of a $\C^1$ curve $c: [0,1] \rightarrow H\o$ as
$$l(c) = \int_{0}^{1} \|\dot{c}(t)\|\ dt.$$
It induces a continuous metric $d_{H\o}$ on $H\o$: the distance between two points $v,w \in H\o$ is the minimal length for $\|\ .\ \|$ of a $\C^1$ curve joining $v$ and $w$.\\
Remark that, if $\o\subset\R\P^2$, then $\|\ .\ \|$ is actually a $C_X$-regular Riemannian metric on $H\o$. When $\o$ is an ellipsoid, we recover the classical Riemannian metric. In any case,  $\|\ .\ \|$ is obviously $J^{X}$-invariant on $h^{X}H\o \oplus VH\o$.\\

\subsection{Stable and unstable distributions}

In \cite{crampon}, I showed why the subbundles $E^u$ and $E^s$ given by
$$E^u=\{Y+J^{X}(Y),\ Y\in VH\o\},\ E^s=\{Y-J^{X}(Y),\ Y\in VH\o\},$$
naturally appeared in the study of the geodesic flow. Recall the

\begin{prop}[\cite{crampon}, Section 4.1 and equation (15)]\label{dd}
$E^u$ and $E^s$ are invariant under the flow, and if $Z^s\in E^s,\ Z^u\in E^u$, then
$$ d\ph^t(Z^u)= e^{t} T^t(Z^u),\ d\ph^t(Z^s)= e^{-t} T^t(Z^s).$$
The operator $J^{X}$ exhanges $E^u$ and $E^s$ and
$$d\ph^t J^{X} (Z^s) = e^{2t} J^{X} (d\ph^t Z^s).$$
\end{prop}
Remark that the second equality is just a consequence of the fact that $J^{X}$ commutes with the parallel transport: we have
$$d\ph^t J^{X} (Z^s) = e^{t} T^t J^{X} (Z^s) =  e^{t} J^{X} T^t(Z^s) = e^{2t} J^{X} (d\ph^t Z^s).$$
The tangent space $TH\o$ splits into
$$TH\o = \R.X \oplus E^s \oplus E^u;$$
this decomposition will be called the \emph{Anosov decomposition}. The main result about the distributions $E^u$ and $E^s$ is the following
\begin{prop}\label{stable}
Let $Z^s\in E^s,\ Z^u\in E^u$. Then $t\longmapsto \|d\ph^t Z^s\|$ is a strictly decreasing bijection from $\R$ onto $(0,+\infty)$, and $t\longmapsto \|d\ph^t Z^u\|$ is a strictly increasing bijection from $\R$ onto $(0,+\infty)$.
\end{prop}

In what follows, the image of a point $w = (x,[\xi]) \in H\o$ under the flow is denoted by $\ph^t(w)= (x_t,[\xi_t])$, for $t\in\R$. We first need a

\begin{lemma}\label{equivalents}
We have
$$\frac{|x_tx^-|}{|x_tx^+|}=e^{2t}\frac{|xx^-|}{|xx^+|}.$$
In particular the following asymptotic expansion holds:
$$|x_tx^+| = \frac{|xx^+|^2}{m(w)}e^{-2t} + O(e^{-4t}).$$
\end{lemma}
\begin{proof} We have $d_{\o}(x,x_t) = t$, which implies
$$e^{2t} = \frac{|xx^-|}{|xx^+|} \frac{|x_tx^-|}{|x_tx^+|},$$
and yields the result.
\end{proof}

In order to make computations easier, we will need the following. {\it A chart adapted to the point $w\in HM$ or to its orbit $\ph.w$} is an affine chart where the intersection $T_{x^+}\partial\o\cap T_{x^-}\partial\o$ is contained in the hyperplane at infinity, and a Euclidean structure on it so that the line $(xx^+)$ is orthogonal to $T_{x^+}\partial\o$ and $T_{x^-}\partial\o$.

\begin{figure}[!h]
\begin{center}
\includegraphics{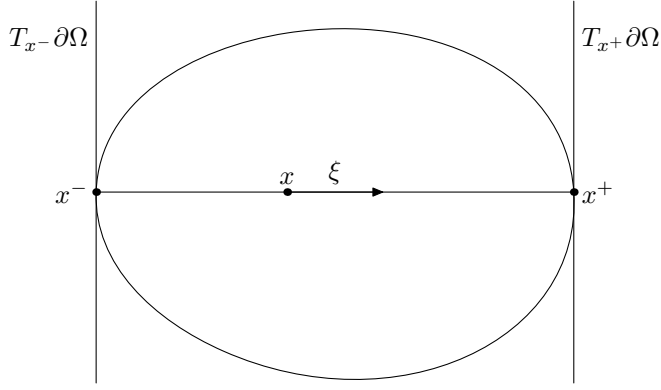}
\end{center}
\caption{A good chart at $w=(x,[\xi])$}
\label{figgoodchart}
\end{figure}

\begin{lemma}\label{transport}
In a good chart at $w=(x,[\xi])$ there exists a constant $C(w)$ such that, for any $Z(w) \in E^s(w)\cup E^u(w)$,
$$\|T^t Z(w)\|= C(w) (|x_tx^+||x_tx^-|)^{1/2}\left(\frac{1}{|x_ty_t^+|}+\frac{1}{|x_ty_t^-|}\right),$$
where $y_t^+$ and $y_t^-$ denote the points of intersection of the line $\{x + \l d\pi(Z(w))\}_{\l\in\R}$ with $\doo$ (see figure \ref{parallelfigure}).
\end{lemma}

\begin{figure}[!h]
\begin{center}
\includegraphics{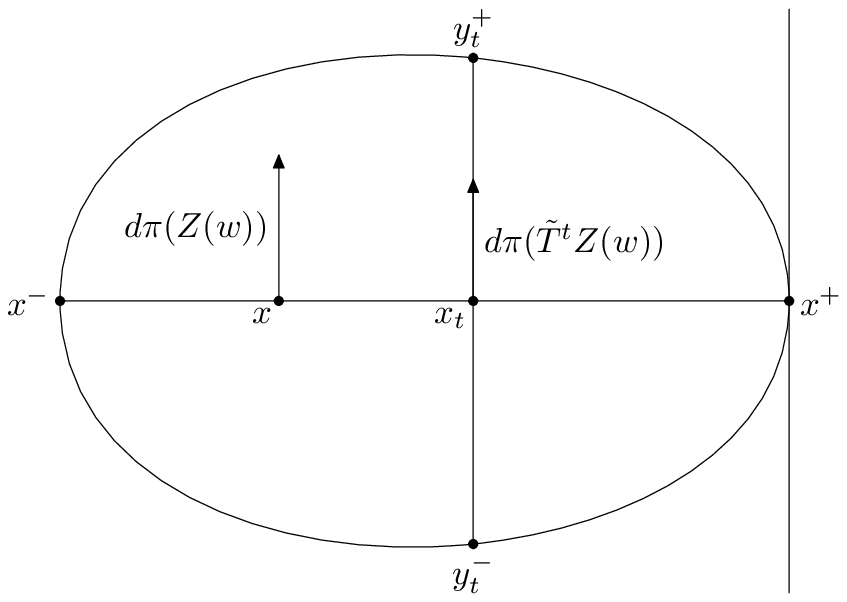}

\end{center}

\caption{Parallel transport on $H\o$}
\label{parallelfigure}
\end{figure}

\begin{proof}
Assume for example that $Z(w)\in E^u(w)$. Then $Z(w) = h(w) + J^X(h(w))$, for some horizontal vector $h(w)$. Let $h$ denote the parallel transport of $h(w)$, which is defined on the orbit $\ph.w$. We have $T^t Z = h + J^X(h)$ on $\ph.w$. In a good chart at $w$, lemma \ref{horver} gives
$$d\pi(h) = (m(w)m)^{1/2}\ d\pi(h^e);$$
in this case, since the chart is adapted, $h^e$ is just the Euclidean parallel transport of $h(w)$ along $\ph.w$. In particular, $|d\pi(h^e)| = |d\pi(h^e(w))|=|d\pi(h(w))|$. Hence
$$\|T^t Z(w)\| = F(d\pi(h(\ph^t w)))= \frac{|d\pi(h(w))| m(w)}{2} m(\ph^t(w))^{1/2}\left(\frac{1}{|x_ty_t^+|}+\frac{1}{|x_ty_t^-|}\right).$$
\end{proof}

We can now give a

\begin{proof}[Proof of proposition \ref{stable}]
Choose a stable vector $Z^s(w)\in E^s(w)$ and a chart adapted to $w=(x,[\xi])$. In that chart, the vector $d\pi(T^t Z^s(w))$ is orthogonal to $x_tx^+$ with respect to the Euclidean structure on the chart; hence so are $x_ty_t^+$ and $x_ty_t^-$.
We have from lemma \ref{dd},
$$\|d\ph^t Z^s(w)\| = e^{-t} \|T^t Z^s(w)\|.$$
Lemma \ref{equivalents} gives $$|x_tx^-| = e^{2t}|x_tx^+|\frac{|xx^-|}{|xx^+|},$$ hence from lemma \ref{transport}, there is a constant $C'(w)$ such that
$$\|d\ph^t Z^s(w)\| = C'(w)\left(\frac{|x_tx^+|}{|x_ty_t^+|}+\frac{|x_tx^+|}{|x_ty_t^-|}\right)$$
The strict convexity of $\o$ implies that the function $h: t \mapsto \frac{|x_tx^+|}{|x_ty_t^+|} + \frac{|x_tx^+|}{|x_ty_t^-|}$ is strictly decreasing on $\R$, the $\C^1$ regularity of $\partial\o$ that $\lim_{t\to +\infty} h(t)=0$ and the strict convexity of $\o$ that $\lim_{t\to +\infty} h(t)=+\infty$.\\

The same computation holds for $t \mapsto \|d\ph^{-t}(Z^u)\|$ for $Z^u\in E^u$.

\end{proof}

\subsection{Horopsheres, stable and unstable manifolds}

\emph{Horospheres} can be defined for \emph{any} Hilbert geometry $(\o,\d)$. Pick a point $x^+\in\doo$. For any point $x\in\o$, call $(xx^+):\R\longrightarrow \o$ the geodesic line such that $(xx^+)(0)=x,\ (xx^+)(+\infty)=x^+$. Given a point $x\in\o$, there is for each point $y\in\o$ a unique time $t_y\in\R$ such that
$$\lim_{t\to+\infty} \d((xx^+)(t), (yx^+)(t+t_0)) = \inf_{z\in(yx^+)} \left\{\lim_{t\to+\infty} \d((xx^+)(t), (zx^+)(t))\right\}.$$
The horosphere $\H_{x^+}(x)$ through $x$ about $x^+$ is the set of such ``minimal points'':
$$\H_{x^+}(x) = \{(yx^+)(t_y),\ y\in\o\}.$$
This is a continuous submanifold of $\o$.\\

Come back now to a strictly convex set $\o$ with $\C^1$ boundary. In this case, as in the hyperbolic space, horospheres can also be defined as level sets of the Busemann functions $b_{x^+}(x,.)$ given by
$$b_{x^+}(x,y) = \lim_{p\to x^+} \d(x,p) - \d(y,p).$$
For $w=(x,[\xi])\in H\o$, let us denote by $\H_w = \H_{x^+}(x)$ the horosphere based at $x^+=\ph^{+\infty}(w)$ and passing through $x$. The horosphere $\H_{\sigma w}$, where $\sigma: (x,[\xi])\in H\o \longmapsto (x,[-\xi])$, is the horosphere $\H_{x^-}(x)$ the horosphere based at $x^-=\ph^{-\infty}(w)$ and passing through $x$.\\
The \emph{stable and unstable manifolds} at $w_0=(x_0,[\xi_0])\in H\o$ are the $\C^1$ submanifolds of $H\o$ defined as
$$W^s(w_0)=\{w=(x,[xw_0^+])\in H\o,\ x\in \H_{w}\},$$
$$W^u(w_0)=\{w=(x,[w_0^-x])\in H\o,\ x\in H_{\sigma w}\}.$$

\begin{figure}[!h]
\begin{center}
\includegraphics{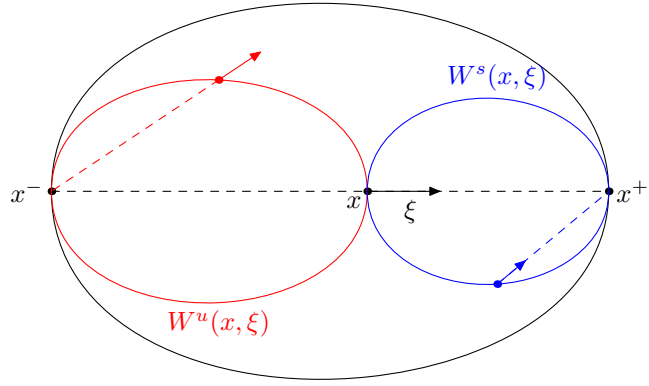}
\end{center}
\caption{Stable and unstable manifolds}
\label{stablemanifold}
\end{figure}

We can check (see \cite{benoistcv1}) that
$$W^s(w_0)=\{w\in H\o,\ \lim_{t\to +\infty} \d(\pi\ph^t(w),\pi\ph^t(w_0))=0\} = \{w\in H\o,\ \lim_{t\to +\infty} d_{H\o}(\ph^t(w),\ph^t(w_0))=0\} ,$$
$$W^u(w_0)=\{w\in H\o,\ \lim_{t\to -\infty} \d(\pi\ph^t(w),\pi\ph^t(w_0))=0\} =\{w\in H\o,\ \lim_{t\to -\infty} d_{H\o}(\ph^t(w),\ph^t(w_0))=0\} .$$
(Recall that $\pi: H\o \longrightarrow \o$ denotes the bundle projection.) As a corollary of proposition \ref{stable}, we have:

\begin{corollary}
The distributions $E^s$ and $E^u$ are the tangent spaces to $W^s$ and $W^u$.
\end{corollary}

\begin{rmk}\label{rmkdistance}
To deduce results on $(\o,\d)$ from results on $(H\o,d_{H\o})$, it is useful to remark that the projection $\pi:H\o\longrightarrow \o$ send isometrically stable and unstable manifolds equipped with the metric induced by $\|\ .\ \|$, on horospheres, with the metric induced by $\d$.
\end{rmk}


\vspace{1cm}

\section{Lyapunov exponents}

The goal now is to understand for a given tangent vector $Z\in TH\o$ the asymptotic behaviour of the norms $\|d\ph^t Z\|$ when $t$ goes to $\pm\infty$. In particular, we want to catch some exponential behaviour by looking at the limits, when they exist,
$$\chi^{\pm}(Z)=\limsup_{t\to \pm\infty} \frac{1}{t} \log \|d\ph^t(Z)\|.$$
When $\chi^{\pm}(Z)\not= 0$, this means that $\|d\ph^t Z\|$ has exponential behaviour when $t\to\pm\infty$: for any $\epsilon>0$, there exists some $C_{\epsilon}>0$ such that, whenever $t>0$,
$$C_{\epsilon}^{-1} e^{\pm(\chi^{\pm}(Z)-\epsilon) t} \leqslant \|d\ph^t(Z)\| \leqslant C_{\epsilon} e^{\pm(\chi^{\pm}(Z)+\epsilon)t}.$$
These two numbers ${\chi}^+(Z)$ and $\chi^-(Z)$ are the forward and backward Lyapunov exponents of the vector $Z$ and they are the main characters of the two next sections.

\subsection{Symmetries}

There are lots of symmetries in our geodesic flow that we should exploit to reduce our study.\\

First, thanks to the Anosov decomposition $TH\o = \R.X \oplus E^s \oplus E^u$, it is enough to study the asymptotic behaviour of the norms $\|d\ph^t Z\|$ for $Z\in E^s$ or $Z\in E^u$; of course, $d\ph^t X=X$, and we can recover the asymptotic behaviour of any vector $Z$ by decomposing it with respect to the Anosov decomposition.\\

Second, thanks to the reversibility of the Hilbert metric, it suffices to study what occurs when $t$ goes to $+\infty$ by using the \emph{flip map}: it is the $\C^{\infty}$ involutive diffeomorphism $\sigma$ defined by
$$\begin{array}{lclc}
\sigma:& H\o & \longrightarrow & H\o\\
& w=(x,[\xi]) & \longmapsto & (x,[-\xi]).
\end{array}$$
The reversiblity of the Hilbert metric implies that $\sigma$ conjugates the flows $\ph^t$ and $\ph^{-t}$:
$$\ph^{-t} = \sigma \circ \ph^t \circ \sigma.$$


\begin{lemma}\label{flipmap}
The differential $d\sigma$ anticommutes with $J^X$, that is, $J^X\circ d\sigma = -d\sigma \circ J^X$. As a consequence, $\sigma$ preserves the decomposition $TH\o  = \R.X \oplus h^XH\o \oplus VH\o$, is a $\|\ .\ \|$-isometry and exchanges stable and unstable distributions and foliations.

\end{lemma}
\begin{proof}
Clearly, $d\sigma(X)=-X$ and $d\sigma$ preserves $VH\o$. Now, just recall how $v_X$ is defined: for any $Y\in VH\o$, we have $v_X(X)=v_X(Y)=0$, and $v_X([X,Y])=-Y$, so
$$d\sigma v_X(X)=v_X(d\sigma(X))=0=d\sigma v_X(Y)=v_X(d\sigma(Y)),$$
and
$$v_X d\sigma([X,Y])=v_X ([d\sigma(X),d\sigma(Y)]) = v_X([-X,d\sigma(Y)] = d\sigma(Y) = -d\sigma v_X([X,Y]).$$
So $d\sigma\circ v_X = - v_X\circ d\sigma$. As for $H_X$:
$$\begin{array}{rl}
d\sigma H_X(Y) = d\sigma (-[X,Y] - \displaystyle\frac{1}{2}v_X[X,[X,Y]]) & = -[d\sigma (X),d\sigma (Y)] + \displaystyle\frac{1}{2}v_X[d\sigma(X),[d\sigma(X),d\sigma(Y)]]\\\\
& = [X,d\sigma(Y)] + \displaystyle\frac{1}{2}v_X[X,[X,d\sigma(Y)]]\\\\
& = -H_X (d\sigma(Y)).
  \end{array}$$
Finally, we get that $d\sigma$ and $J^X$ anticommute. This implies in particular that $\sigma$ preserves the horizontal bundle $h^XH\o$ and the metric $\|\ .\ \|$. It also gives that, if $Z=Y+J^X(Y)\in E^u$, then $d\sigma(Z) = d\sigma(Y) - J^X d\sigma (Y)\in E^s$, hence $d\sigma(E^u)=E^s$, and conversely; so $\sigma$ exchanges stable and unstable distributions and foliations.\\
\end{proof}

Now, since $\ph^{-t} = \sigma \circ \ph^t \circ \sigma$, we have
$$\limsup_{t\to -\infty} \frac{1}{t} \log \|d \ph^t Z\|= \limsup_{t\to-\infty} \frac{1}{t} \log \|d\sigma d\ph^{-t} d\sigma Z\|=\limsup_{t\to+\infty} \frac{1}{-t} \log \|d\ph^{t} d\sigma Z\|$$
because $\sigma$ preserves $\|\ .\ \|$. Hence
\begin{equation}\label{pastfuture}
 \limsup_{t\to -\infty} \frac{1}{t} \log \|d \ph^t Z\|=-\liminf_{t\to+\infty} \frac{1}{t} \log \|d\ph^{t} d\sigma Z\|.
\end{equation}
This equality allows us to deduce the behaviour in the future from the one in the past: to catch the behaviour of stable vectors in the past, one can study the behaviour of unstable vectors in the future; and conversely.\\

Finally the operator $J^X$ provides a symmetry between $E^u$ and $E^s$: it sends the stable vector $Z^s = Y - J^X(Y)$ to the unstable vector $Z^u = Y + J^X(Y)$. Furthermore, since $J^X$ commutes with $T^t$ and $T^t$ preserves horizontal and vertical distributions, we have
$$d\ph^t Z^u = e^{t} T^t Z^u = e^t(T^t Y - J^X (T^t Y)) ,\ d\ph^t Z^s = e^{-t} T^t Z^s=e^{-t}(T^t Y - J^X (T^t Y)).$$
But, from the very definition of the metric $\|\ .\ \|$, we have
$$\|T^t Z^u\| = \|T^t Z^s\| = F(d\pi J^X(T^t Y)).$$

To understand the asymptotic behaviour of the norms $\|d\ph^t Z^u\|$ and $\|d\ph^t Z^s\|$, it then suffices to understand the behaviour of the quantities $F(d\pi T^t h)$ for $h\in h^XH\o$ (recall that $J^X$ exchanges $VH\o$ and $h^XH\o$, so that $J^X(T^t Y)\in h^XH\o$). This is what we will do in the next part.

\subsection{Parallel Lyapunov exponents}

Remark that, given a point $w = (x,[\xi])\in H\o$, the projection of the horizontal space $h_w^XH\o$ at $w$ on $T\o$ is precisely the tangent space $T_x\H_w$ to the horosphere $\H_w$ at the point $x$. We now define a parallel transport along oriented geodesics on $\o$ that will contain all the information we need and become the main object of our study.\\

Let fix a point $x^+\in\doo$. Denote by $W^s(x^+)=\{w\in H\o,\ \ph^{+\infty}(w)=x^+\}$ the weak stable manifold associated to $x^+$, consisting of these points $w$ that end at $x^+$. Obviously, the map $\pi$ identifies $W^s(x^+)$ with $\o$, and we will call $\pi^{-1}_{x^+}$ the inverse of $\pi_{|_{W^s(x^+)}}$; we have $\pi^{-1}_{x^+}(x) = (x,[xx^+])$.\\
The radial flow $\ph_{x^+}^t$ is the flow on $\o$ defined via
$$\ph_{x^+}^t = \pi\ph^t\pi^{-1}_{x^+}.$$
It is generated by the vector field $X_{x^+}$ such that $[X_{x^+}]=[xx^+]$ and $F(X_{x^+})=1$. Obviously, this flow preserves the set $\{\H_w,\ w\in W^s(x^+)\}$ of horospheres based at $x^+$, by sending $\H_w$ on $\H_{\ph^t(w)}$; also it contracts the Hilbert distance $\d$. Finally, the space $T\o$ admits a $\ph_{x^+}^t$-invariant decomposition
$$T\o = \R.X_{x^+} \oplus T\H_{x^+},$$
where $T\H_{x^+}$ is the bundle over $\o$ defined as
$$T\H_{x^+} = \{T_x\H_w,\ w=(x,[\xi])\in W^s(x^+)\}.$$
Furthermore, from the very definition of the radial flow, we have $d\ph_{x^+}^t = d\pi d\ph^t d\pi^{-1}_{x^+}$; so, for any vector $v\in T\H_{x^+}$, we have
$$d\ph^t_{x^+}(v) = d\pi d\ph^t  d\pi^{-1}_{x^+}(v),$$
where $d\pi^{-1}_{x^+}(v)$ is a stable vector. The action of $d\ph^t$ on $E^s$ can be deduced from the action of the parallel transport on $E^s$, and we now define a parallel transport on $\o$ to get the same kind of relations.

\begin{defi}
Let $x^+\in\doo$. The parallel transport $T_{x^+}^t$, $t\in\R$, in the direction of $x^+$ is defined by
$$T_{x^+}^t = d\pi T^t  d\pi^{-1}_{x^+}.$$
\end{defi}

Given a vector $v\in T\o$, we deduce its parallel transport $T_{x^+}^t(v)$ by taking the unique vector $Z(v) \in E^s \oplus \R.X$ that projects on $v$, take its parallel transport $T^t Z(v)$ and project it again. Equivalently, since $E^s=\{Y-J^X(Y),\ Y\in VH\o\}$, we can also take the unique  vector $h(v)$ in $\R.X \oplus h^XH\o$ that projects on $v$.\\
From proposition \ref{stable}, we deduce that, for any $v\in T\H_{x^+}$,
\begin{equation}
d\ph_{x^+}^t(v) = e^{-t} T_{x^+}^t(v)
\end{equation}

The only thing we have to do now is to understand the behaviour of the quantities $F(T_{x^+}^t v)$ for $v\in T\H_{x^+}$.

\begin{defi}
Let $x^+\in\doo$. The \emph{upper and lower parallel Lyapunov exponents} $\overline{\eta}(x^+,v)$ and $\overline{\eta}(x^+,v)$ of a vector $v\in T\H_{x^+}$ in the direction of $x^+$, are defined as
$$\overline{\eta}(w,v)=\limsup_{t\to+\infty} \frac{1}{t} \log F( T_{x^+}^t v),\ \underline{\eta}(w,v)=\liminf_{t\to+\infty} \frac{1}{t} \log F( T_{x^+}^t v).$$
\end{defi}

Given $w=(x,[\xi])\in W^s(x^+)$, it is not difficult to see that the numbers $\overline{\eta}(x^+,v)$ and $\underline{\eta}(x^+,v)$ can take only a finite number $\overline{p}(w)$ and $\underline{p}(w)$ of values when $v$ describes $T_x\H_w$. More precisely, there exist a $\ph^t_{x^+}$-invariant filtration
$$\{0\}=\overline{H}_{0} \varsubsetneq \overline{H}_{1} \varsubsetneq \cdots \varsubsetneq \overline{H}_{\overline{p}(w)} = T_w \H_w$$
along the orbit $\ph_{x^+}.x$, and real numbers
$$\overline{\eta}_{1}(w) < \cdots < \overline{\eta}_{\overline{p}(w)}(w),$$
called upper parallel Lyapunov exponents,
such that for any vector $v_i\in \overline{H}_{i}\smallsetminus \overline{H}_{i-1}$, $1\leqslant i\leqslant \overline{p}(w)$,
$$\limsup_{t\to\pm} \frac{1}{t} \log F(T_{x^+}^t v_i) = \overline{\eta}_{i}(w).$$
The same occurs for lower parallel Lyapunov exponents.\\

As a consequence of part \ref{sectionlyapunovboundary}, we will have the following

\begin{corollary}
Let $x^+\in\doo$. The numbers $\overline{p}$ and $\underline{p}$ are constants on $W^s(x^+)$, as well as the numbers $\overline{\eta}_i,\ 1\leqslant i\leqslant \overline{p}$ and $\underline{\eta}_i,\ 1\leqslant i\leqslant \underline{p}$.
\end{corollary}

\subsection{Regular points}\label{sectionregular}

Recall the following general

\begin{defi}\label{regular}
Let $\ph^t$ be a $\C^1$-flow on a Finsler manifold $(W,\|\ . \ \|)$. A point $w\in W$ is said to be \emph{regular} if there exist a $\ph^t$-invariant decomposition
$$TW = E_1 \oplus \cdots \oplus E_p,$$
along the orbit $\ph.w$, called Lyapunov decomposition, and real numbers
$$\chi_1(w) < \cdots < \chi_{p}(w),$$
called Lyapunov exponents, such that, for any vector $Z_i\in  E_i\smallsetminus \{0\}$,
\begin{equation}\label{regularg}\lim_{t\to\pm\infty} \frac{1}{t} \log \|d\ph^t(Z_i)\| = \chi_i(w),\end{equation}
and
\begin{equation}\label{regulardet} \lim_{t\to\pm\infty} \frac{1}{t} \log |\det d\ph^t| = \sum_{i=1}^p \dim E_i\ \chi_i(w).\end{equation}
The point $w$ is said to be \emph{forward} or \emph{backward regular} if this behaviour occurs only when $t$ goes respectively to $+\infty$ or $-\infty$.
\end{defi}

In this definition, we have to precise what is meant by $\det d\ph^t$, since $\|\ .\ \|$ is not a Riemannian metric. The determinant $\det d\ph^t$ just measures the effect of $\ph^t$ on volumes. But associated to the Finsler metric $\|\ .\ \|$ is the Busemann volume $vol_{W}$, which is the volume form defined by saying that, in each tangent space $T_wW$, the volume of the unit ball of $\|\ .\ \|$ is the same as the volume of the Euclidean unit ball of the same dimension. In other words, given an arbitrary Riemannian metric $g$ on $W$ with Riemannian volume $vol_g$, we have, at the point $w \in W$,
$$dvol_{W}(w) = \frac{vol_g(B_g(w,1))}{vol_g(B(w,1))} dvol_g(w),$$
where $B(w,1)$ and $B_g(w,1)$ denote the unit balls in $T_wW$ for, respectively, $\|\ .\ \|$ and $g$. The determinant $\det d\ph^t$ is then defined in this way: if $A$ is some Borel subset of $T_wW$ with non-zero volume, then
$$|\det d_w\ph^t | = \frac{vol_{W}(\ph^t(w)) (d\ph^t A)}{vol_{W}(w)(A)}.$$

\ \\

Let us specify what happens in our case at a regular point $w\in H\o$. First, it has always $0$ as Lyapunov exponent since $\|d\ph^t(X)\|=1$, and we will say that $w$ has \emph{no zero Lyapunov exponent} if the subspace $E_0$ corresponding to the exponent $0$ is $E_0=\R.X$.\\
Second, proposition \ref{stable} implies that $\chi(Z^s)\leqslant 0$ and $\chi(Z^u)\geqslant 0$ for any $Z^s\in E^s(w),\ Z^u\in E^u(w)$. Furthermore, if $Z^s\in E^s(w)$ and $Z^u=J^X Z^s$ is the corresponding vector of $E^u(w)$, proposition \ref{stable} gives
$$\chi(Z^u) = 2 + \chi(Z^s).$$
Now, choose a tangent vector $Z$ whose Lyapunov exponent is $0$. $Z$ can be written as $Z=aX+Z^u+Z^s$ for some $a\in\R,\ Z^s\in E^s,\ Z^u\in E^u$. Since
$$\lim_{t\to\pm\infty} \frac{1}{t} \log \|d\ph^t(Z)\| = 0,$$
we conclude that $\chi(Z^u)=\chi(Z^s)=0$. Thus, the subspace $E_0$ corresponding to the Lyapunov exponent $0$ can be decomposed as
$$E_0 = \R.X \oplus E^- \oplus E^+,$$
where $E^- \subset E^s,\ E^+ \subset E^u$.\\

At a regular point, the Lyapunov decomposition can thus be written in the following way:
\begin{equation}\label{oseledetsdecomposition}
  TH\o = E^s_0 \oplus (\oplus_{i=1}^p E^s_i) \oplus E_{p+1}^s \oplus \R.X \oplus E_0^u \oplus (\oplus_{i=1}^p E^u_i) \oplus E^u_{p+1},
\end{equation}
with the relations
$$E_i^s=J^X(E_i^u),\ 0\leqslant i\leqslant p+1.$$
The subspaces $E^s_0$ and $E^u_0$, or $E_{p+1}^s$ and $E_{p+1}^s$, might be $\{0\}$. The corresponding Lyapunov exponents are
$$-2=\chi^-_0 < \chi_1^s < \cdots <\chi^s_p < \chi^s_{p+1} = 0 = \chi^u_0  < \chi^u_1 <\cdots<\chi^u_p < \chi^u_{p+1} = 2,$$
with
$$\chi_i^u=\chi_i^s+2,\ 0\leqslant i\leqslant p.$$
If $w$ has no zero Lyapunov exponent then $E^s_0=E^u_0=E_{p+1}^s=E_{p+1}^s=\{0\}$ and all the Lyapunov exponents are {\it strictly} between $-2$ and $2$.\\

If we now look down at the base manifold $\o$, we see that, if $w=(x,[\xi])\in H\o$ is a regular point ending at $x^+\in\doo$, the decomposition (\ref{oseledetsdecomposition}) induces by projection a decomposition
$$T\o = \R.X_{x^+} \oplus H_0 \oplus H_1 \oplus \cdots \oplus \cdots \oplus H_p \oplus H_{p+1}$$
along the orbit $\ph_{x^+}.x$ and there exist real numbers $-1=\eta_0<\eta_1 < \cdots < \eta_p<\eta_{p+1}=1$, that we call parallel Lyapunov exponents, such that, for any vector $v_i \in  H_i\smallsetminus \{0\}$,
$$\lim_{t\to+\infty} \frac{1}{t} \log F(T^t_{x^+}(v_i)) = \eta_i,$$
and
$$ \lim_{t\to+\infty} \frac{1}{t} \log |\det T_{x^+}^t| = \sum_{i=1}^p \dim H_i\ \eta_i.$$

We have $\R.X_{x^+}=d\pi(\R.X)$ and
\begin{equation}\label{dechautbas}
 H_i = d\pi (E_i^s) = d\pi(E_i^u);
\end{equation}
in particular, $H_0$ and $H_{p+1}$ can be $\{0\}$. Also, the parallel Lyapunov exponents $\eta_i$ are given by
\begin{equation}\label{lyaphautbas}
\eta_i = \chi^s_i+1 = \chi^u_i-1.
\end{equation}

We then have the following characterization of regular points:

\begin{prop}
Let $x^+\in\doo$. A point $w=(x,[\xi])\in W^s(x^+)$ is regular for $\ph^t$ if and only if the point $x$ is regular for $\ph^t_{x^+}$. The decomposition and Lyapunov exponents are linked by the relations (\ref{dechautbas}) and (\ref{lyaphautbas}).
\end{prop}

Obviously, all of this makes sense for forward and backward regular points.

\subsection{Oseledets theorem}

The essential result about regular points is the following theorem of Oseledets, which, given an invariant probability measure of the flow, gives a condition for almost all points to be regular.

\begin{thm}[Osedelets' ergodic multiplicative theorem \cite{osedelec}]\label{fullmeasure}
Let $\ph^t$ be a $\C^1$ flow on a Finsler manifold $(W,\|\ .\ \|)$ and $\mu$ a $\ph^t$-invariant probability measure. If
\begin{equation}\label{hyposeledets}
\frac{d}{dt}_{|_{t=0}} \log \|d\ph^{\pm t}\| \in L^1(W,\mu),
\end{equation}
then the set of regular points has full measure.
\end{thm}

Assumption $(\ref{hyposeledets})$ means that the flow does not expand or contract locally too fast. This essentially allows us to use Birkhoff's ergodic theorem to prove the theorem.\\
The next lemma proves that our geodesic flow satisfies assumption $(\ref{hyposeledets})$. Obviously, Oseledets' theorem is not interesting on $H\o$ itself since there is no finite invariant measure. But it can be applied for any invariant measure of the geodesic flow of a given a quotient manifold $M=\o/\G$.\\

Remark that our Finsler metric is $C_X$-regular so condition $(\ref{hyposeledets})$ makes sense. Furthermore, Oseledets' theorem is usually stated on a Riemannian manifold but it is still valid for a Finsler one: using John's ellipsoid, it is always possible to define a Riemannian metric $\|\ .\ \|_{J}$  which is bi-Lipschitz equivalent to $\|\ .\ \|$, that is, such that
$$\frac{1}{\sqrt{n}} \|Z\|_{J} \leqslant \|Z\| \leqslant \sqrt{n} \|Z\|_{J},\ Z\in TW$$
where $n$ is the dimension of the manifold; of course, there is no reason for this metric $\|\ .\ \|_J$ to be even continuous but it is not important.

\begin{lemma}
For any  $Z^s\in E^{s},\ Z^u\in E^{u}$, we have
$$-2 \leqslant \frac{d}{dt}_{|_{t=0}} \|d\ph^{t} Z^s\|\log  \leqslant 0 \leqslant \frac{d}{dt}_{|_{t=0}}  \|d\ph^{t} Z^u\| \leqslant 2.$$
In particular, for any $t\in\R$ and $Z\in TH\o$,
$$e^{-2|t|} \|Z\| \leqslant \|d\ph^t(Z)\| \leqslant e^{2|t|} \|Z\|.$$
\end{lemma}

This lemma clearly implies the already known fact (coming from proposition \ref{stable}) that Lyapunov exponents are all between $-2$ and $2$. But it is more precise: it gives that the rate of expansion/contraction is {\it at any time} between $-2$ and $2$, not only asymptotically, and that is what is essential to apply Oseledets' theorem.

\begin{proof}
It is a direct corollary of proposition \ref{stable}: we know that  $t\mapsto \|d\ph^t Z^s\|$ is decreasing, hence
$$\lim_{t\to 0}\frac{1}{t} \log \|d\ph^t Z^s\| \leqslant 0.$$ But we also know from proposition \ref{dd} that
$$\|d\ph^t Z^s\| = e^{-2t} \|d\ph^t J^X(Z^s)\|.$$
Since $J^X(Z^s) \in E^u$, proposition \ref{stable} tells us that  $t\mapsto \|d\ph^t J^X(Z^s)\|$ is increasing, hence
$$\lim_{t\to 0} \frac{1}{t} \log \|d\ph^t J^X(Z^s)\| \geqslant 0$$
 and
$$\lim_{t\to 0} \frac{1}{t}  \log \|d\ph^t Z^s\|\geqslant -2.$$
Using $J^X$, we get the second inequality, and by integrating, we get the last one.
\end{proof}


\vspace{1cm}

\section{Structure of the boundary $\doo$}\label{sectionlyapunovboundary}

In this part, we give a link between parallel Lyapunov exponents and the shape of the boundary at the endpoint of the orbit.\\

\subsection{Motivation}

We first give the idea in dimension $2$. Let $x^+\in\doo$, $w = (x,[\xi])\in W^s(x^+)$ and choose a vector $v$ tangent to $\H_w$ at $x$, with parallel Lyapunov exponent $\eta$.
In a good chart at $w$, lemma \ref{transport} gives
$$F(T_{x^+}^t v) = C(w) (|x_tx^+||x_tx^-|)^{1/2}\left(\frac{1}{|x_ty_t^+|}+\frac{1}{|x_ty_t^-|}\right).$$

Assume that $|x_ty_t^-|\asymp |x_ty_t^+|$. Then
$$\lim_{t\to +\infty} \frac{1}{t} \log \frac{F( T_{x^+}^t v)}{|x_tx^+|^{1/2}} = -\lim_{t\to +\infty} \frac{1}{t}\log |x_ty_t^+|,$$
hence, dividing by $\log |x_tx^+|^{1/2}$,
$$\lim_{t\to +\infty}\frac{\log F(  T_{x^+}^t )}{\log |x_tx^+|^{1/2}} -1 = -\lim_{t\to +\infty}\frac{\log |x_ty_t^+|}{\log |x_tx^+|^{1/2}}.$$
Since $|x_tx^+|\asymp e^{-2t}$, that yields
$$\lim_{t\to +\infty}\frac{\log |x_ty_t^+|}{\log |x_tx^+|} = \frac{1+\eta}{2}.$$
Let $f: T_{x^+}\doo \longrightarrow \R^n$ be the graph of $\doo$ at $x^+$, so that $|x_tx^+|= f(|x_ty_t^+|)$. We thus obtain
$$\lim_{s\to 0} \frac{\log f(s)}{\log s} = \frac{2}{1+\eta},$$
that is, for any $\epsilon>0$, there exists $C>0$ such that
\begin{equation}\label{eeee}
  C^{-1} s^{\frac{2}{1+\eta}+\epsilon}\leqslant f(s) \leqslant C s^{\frac{2}{1+\eta}-\epsilon}.
\end{equation}

This link was first established in \cite{crampon} for divisible convex sets, where the condition $|x_ty_t^-|\asymp |x_ty_t^+|$ is always satisfied. In order to generalize it, we must introduce new definitions. It may be a bit fastidious so you could prefer going directly to proposition \ref{linkboundary}, and have a look to the part in between when it is needed.

\subsection{Locally convex submanifolds of $\R\P^n$}

\begin{defi}
A codimension 1 $\C^0$ submanifold $N$ of $\R^n$ is \emph{locally (strictly) convex} if for any $x\in N$, there is a neighbourhood $V_x$ of $x$ in $\R^n$ such that $V_x\smallsetminus N$ consists of two connected components, one of them being (strictly) convex. \\
A codimension 1 $\C^0$ submanifold $N$ of $\R\P^n$ is \emph{locally} (strictly) \emph{convex} if its trace in any affine chart is locally (strictly) convex.
\end{defi}

Obviously, to check if $N\subset \R\P^n$ is convex around $x$, it is enough to look at the trace of $N$ in {\it one} affine chart at $x$.
Choose a point $x\in N$ in a locally convex submanifold $N$ and an affine chart centered at $x$. We can find a tangent space $T_x$ of $N$ at $x$ such that $V_x\cap N$ is entirely contained in one of the closed half-spaces defined by $T_x$. We can then endow the chart with a suitable Euclidean structure, so that, around $x$, $N$ appears as the graph of a convex function $f: U\subset T_x \longrightarrow [0,+\infty)$ defined on an open neighbourhood $U$ of $0\in T_x$. This function is (at least) as regular as $N$, is positive, $f(0)=0$ and $f'(0)=0$ if $N$ is $\C^1$ at $x$.  When $N$ is strictly locally convex, then $f$ is strictly convex, in particular $f(v)>0$ for $v\not= 0$.\\

In what follows, we are interested in the shape of the boundary $\doo$ of $\o$ at some specific point, or, more generally, in the local shape of locally strictly convex $\C^1$ submanifolds of $\R\P^n$. Denote by $\mathtt{Cvx}(n)$ the set of strictly convex $\C^1$ functions $f: U\subset\R^n \longrightarrow \R$ such that $f(0)=f'(0)=0$, where $U$ is an open convex subset of $\R^n$ containing $0$. We look for properties of such functions at the origin which are invariant by projective transformations.\\

\subsection{Approximate $\alpha$-regularity}

We introduce here the main notion of approximate $\alpha$-regularity, describe its meaning and prove some useful lemmas.

\subsubsection{Definition}

\begin{defi}\label{defiappreg}
A function $f\in\mathtt{Cvx}(1)$ is said to be \emph{approximately $\alpha$-regular}, $\alpha\in [1,+\infty]$, if
$$\lim_{t\to 0} \frac{\log \displaystyle\frac{f(t)+f(-t)}{2}}{\log |t|} = \alpha.$$
\end{defi}


This property is clearly invariant by affine transformations, and in particular by change of Euclidean structure. It is in fact invariant by projective ones, but we do not need to prove it directly, since it will be a consequence of proposition \ref{linkboundary}.\\
Obviously, the function $t\in\R \mapsto |t|^{\alpha}$, $\alpha>1$ is approximately $\alpha$-regular. To be $\alpha$-regular, with $1<\alpha<+\infty$, means that we roughly behave like $t\mapsto |t|^{\alpha}$. \\
The case $\alpha=\infty$ is a particular one: $f$ is $\infty$-regular means that for any $\alpha\geqslant 1$, $f(t) \ll |t|^\alpha$ for small $|t|$. An easy example of such a function is provided by $f:t \longmapsto e^{-1/t^2}$. On the other side, $f$ is $1$-regular means that for any $\alpha>1$, $f(t) \gg |t|^\alpha$. An example of function which is $1$-regular is provided by the Legendre transform of the last one (see section \ref{sectionlegendre}).\\

In the case where $1<\alpha<+\infty$, we can state the following equivalent definitions. The proof is straightforward.

\begin{lemma}\label{calpha}
Let $f\in\mathtt{Cvx}(1)$ and  $1<\alpha<+\infty$. The following propositions are equivalent:
\begin{itemize}
 \item $f$ is approximately $\alpha$-regular;
 \item for any $\epsilon>0$ and small $|t|$,
$$|t|^{\alpha+\epsilon} \leqslant \frac{f(t)+f(-t)}{2} \leqslant |t|^{\alpha-\epsilon};$$
 \item the function $t\longmapsto \displaystyle\frac{f(t)+f(-t)}{2}$ is $\C^{\alpha-\epsilon}$ and $\alpha+\epsilon$-convex at $0$ for any $\epsilon>0$.
\end{itemize}
\end{lemma}

To understand the last proposition, we recall the following

\begin{defis}\label{deficalpha}
Let $\alpha,\beta\geqslant 1$ We say that a function $f\in\mathtt{Cvx}(n)$ is
\begin{itemize}
 \item $\C^{\alpha}$ if for small $|t|$, there is some $C>0$ such that
$$f(t) \leqslant C|t|^{\alpha};$$
 \item $\beta$-convex if for small $|t|$, there is some $C>0$ such that
$$f(t) \geqslant C|t|^{\beta}.$$
\end{itemize}
\end{defis}

\subsubsection{A useful equivalent definition}

We now give another equivalent definition of approximate regularity, that shows the relation with the motivation above. Theorem \ref{linkboundary} is the most important consequence of it.\\

Let $f\in\mathtt{Cvx}(1)$. Denote by $f^+=f_{|_{[0,1]}}^{-1}$ and $f^-=-f_{|_{[-1,0]}}^{-1}$. These functions are both nonnegative, increasing and concave and their value at $0$ is $0$; they are $\C^1$ on $(0,1]$ and their tangent at $0$ is vertical.\\

The harmonic mean of two numbers $a,b >0$ is defined as
$$H(a,b) = \frac{2}{a^{-1} + b^{-1}}.$$
The harmonic mean of two functions $f,\ g: X \to (0,+\infty)$ defined on the same set $X$ is the function $H(f,g)$ defined for $x\in X$ by
$$H(f,g)(x) = H(f(x),g(x)) = \frac{2}{\frac{1}{f(x)}+\frac{1}{g(x)}}.$$

\begin{prop}\label{approxeq}
A function $f\in\mathtt{Cvx}(1)$ is approximately $\alpha$-regular, $\alpha \in [1,+\infty]$ if and only if
$$\lim_{t\to 0^+} \frac{\log H(f^+,f^-)(t)}{\log t} = \alpha^{-1},$$
with the convention that $\frac{1}{+\infty} = 0$.
\end{prop}
\begin{proof}
As we will see, it is enough to take $f$ continuous, so by replacing $f^+$ and $f^-$ by $\min(f^+,f^-)$ and $\max(f^+,f^-)$, we can assume that $f^+\leqslant f^-$, that is $f(t)\geqslant f(-t)$ for $t\geqslant 0$. Now, assuming that the limit exists,

$$  \lim_{t\to 0^+} \frac{\log H(f^+,f^-)(t)}{\log t} =- \lim_{t\to 0^+} \frac{\log \left(\displaystyle\frac{1}{f^+(t)} + \displaystyle\frac{1}{f^-(t)}\right)}{\log t}
= \lim_{t\to 0^+} \frac{\log f^+(t)}{\log t} - \lim_{t\to 0^+}\frac{\log \left(1 + \displaystyle\frac{f^+(t)}{f^-(t)}\right)}{\log t}.
 $$
Since $f^+\leqslant f^-$, the second limit is $0$, and the first one is
$$\lim_{t\to 0^+} \frac{\log f^+(t)}{\log t} = \lim_{u\to 0^+} \frac{\log u}{\log f(u)}.$$

But, since $f(u)\geqslant f(-u)$ for $u\geqslant 0$, we get
$$\lim_{u\to 0^+} \frac{\log u}{\log \frac{f(u)+f(-u)}{2}} = \lim_{u\to 0^+} \frac{\log u}{\log f(u) + \log \left( 1+\frac{f(-u)}{f(u)}\right)} = \lim_{u\to 0^+} \frac{\log u}{\log f(u)},$$
hence the result.
\end{proof}








The last construction can be generalized in a way that will be useful later, for proving proposition \ref{linkboundary}. Let $f\in\mathtt{Cvx}(1)$ and pick $a>0$. We define two new ``inverse functions'' $f_a^+(s)$ and $f^-_a(s)$ for $s\in[0,\epsilon]$, $\epsilon>0$ small enough, depending on $a$; these are positive functions defined by the equations
$$f(f^+_a(s))=s-sf^+_a(s); f(-f^-_a(s))=s+sf^-_a(s).$$

\begin{figure}[!h]\label{figureapprox}
\begin{center}
\includegraphics{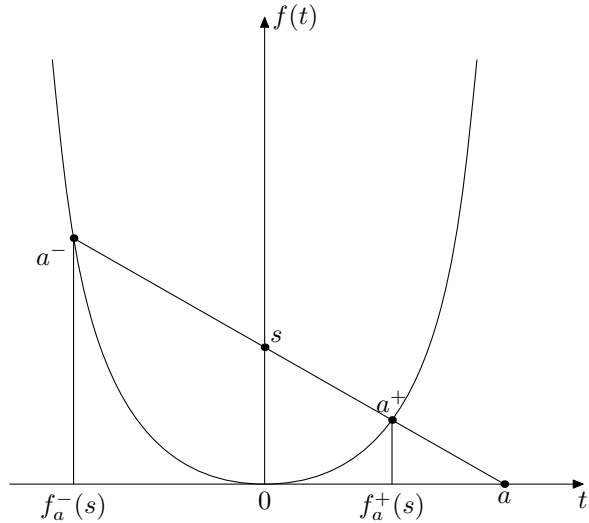}
\end{center}
\caption{Construction of new inverses}
\end{figure}

Geometrically, for $s\in [0,\epsilon]$ on the vertical axis, the line $(as)$ cuts the graph of $f$ at two points $a^+$ and $a-$, with $s$ between $a^+$ and $a^-$; $f_a^+(s)$ and $f^-_a(s)$ are the abscissae of $a^+$ and $a^-$ (c.f. figure \ref{figureapprox}). $f^+$ and $f^-$ can be considered as $f_{+\infty}^+$ and $f_{+\infty}^-$.

\begin{lemma}\label{bordtordu}
Let $f\in\mathtt{Cvx}(1)$ and $a>0$. The functions $\displaystyle\frac{f_a^+}{f^+}$ and $\displaystyle\frac{f_a^-}{f^-}$ can be extended by continuity at $0$ by
$$\frac{f_a^+}{f^+}(0)=\frac{f_a^-}{f^-}(0)=1.$$ In particular, for $s>0$ small enough,
$$ f^+(s) \asymp f_a^+(s),\ f^-(s) \asymp f_a^-(s).$$
\end{lemma}
\begin{proof}
We prove it for $f^+$ and $f_a^+$. Clearly, we have $\frac{f_a^+(s)}{f^+(s)} \leqslant 1$. Since $f$ is convex and $f(0)=0$, we get
$$s-sf_a^+(s) = f(f_a^+(s)) = f\left(\frac{f_a^+(s)}{f^+(s)} f^+(s)\right) \leqslant \frac{f_a^+(s)}{f^+(s)} f(f^+(s)) = \frac{f_a^+(s)}{f^+(s)}s.$$
Hence, for $0<s\leqslant \epsilon<1$
$$ \frac{f_a^+(s)}{f^+(s)} \geqslant 1-f_a^+(s) \geqslant 1 - f_a^+(\epsilon).$$
The function $\displaystyle \frac{f_a^+}{f^+}$ can even be extended at $0$ by $\displaystyle\frac{f_a^+}{f^+}(0)=1$
\end{proof}

The result to remember is the following consequence of lemmas \ref{bordtordu} and \ref{approxeq}:

\begin{corollary}\label{corotordu}
Pick $a>0$. A function $f\in\mathtt{Cvx}(1)$ is approximately $\alpha$-regular if and only if
$$\lim_{t\to 0^+} \frac{\log H(f_a^+,f_a^-)(t)}{\log t} = \alpha^{-1}.$$
\end{corollary}

\subsection{Higher dimensions}

We end this section by extending the definitions in higher dimensions:

\begin{defis}
A function $f\in\mathtt{Cvx}(n)$ is said to be \emph{approximately regular} at $x$ if it is approximately regular in any direction, that is, for any  $v\in \R^n\smallsetminus\{0\}$, there exists $\alpha(v)\in [1,\infty]$ such that
$$\lim_{t\to 0} \frac{\log \displaystyle\frac{f(tv) + f(-tv)}{2}}{\log |t|} = \alpha(v).$$
\end{defis}

Let $f\in\mathtt{Cvx}(n)$ . The upper and lower Lyapunov exponents $\overline{\alpha}(v)$ and $\underline{\alpha}(v)$ of $v\in \R^n$ are defined by

$$\overline{\alpha}(v) = \limsup_{t\to 0} \frac{\log \displaystyle\frac{f(tv) + f(-tv)}{2}}{\log |t|},$$

$$\underline{\alpha}(v) = \liminf_{t\to 0} \frac{\log \displaystyle\frac{f(tv) + f(-tv)}{2}}{\log |t|}.$$

The function is then approximately regular if and only if the preceding limits are indeed limits in $[1,+\infty]$, that is, for any $v\in \R^n$,
$\overline{\alpha}(v)=\underline{\alpha}(v)$. Obviously, lemma \ref{approxeq} and corollary \ref{corotordu} have their counterpart in higher dimensions.

\subsection{Approximate regularity of the boundary}

If $\o$ is a bounded convex set in the Euclidean space $\R^n$ with $\C^1$ boundary, the graph of $\doo$ at $x$ is the function
$$
\begin{array}{rrcl}
  f: & U\subset T_x\doo & \longrightarrow & \R^n \\
	  &  u      & \longmapsto     & \{u + \l n(x)\}_{\l\in\R} \cap \doo,
\end{array}
$$
where $n(x)$ denotes a normal vector to $\doo$ at $x$, and $U$ is a sufficiently small open neighbourhood of $x\in\doo$ for the function to be defined.\\

We can now state our main result. Let $x^+\in\doo$. If $w=(x,[\xi])\in W^s(x^+)$ and $v\in T_x\H_w$, we denote by $p_{x^+}(v)$ the projection of $v$ on the space $T_{x^+}\doo$ in the direction $[xx^+]$. The map $p_{x^+}$ clearly induces an isomorphism $p_{x^+}(x)$ between each $T_x\H_w$ and  $T_{x^+}\doo$.


\begin{thm}\label{linkboundary}
Let $\o$ be a strictly convex proper open set of $\R\P^n$ with $\C^1$ boundary. Pick $x^+\in\doo$, choose any affine chart containing $x^+$ and a Euclidean metric on it.\\
Then for any $v\in T\H_{x^+}$, we have

$$\overline{\eta}(x^+,v)= \frac{2}{\underline{\alpha}(x^{+},p_{x^+}(v))}-1, \ \underline{\eta}(x^+,v)= \frac{2}{\overline{\alpha}(x^{+},p_{x^+}(v))}-1,$$
where $\underline{\alpha}(x^{+},p_{x^+}(v))$ and $\overline{\alpha}(x^{+},p_{x^+}(v))$ denote the lower and upper Lyapunov exponents of the graph of $\doo$ at $x^+$ in the direction $p_{x^+}(v)$, as defined at the very end of the last section.
\end{thm}

\begin{figure}[!h]\label{bord}
\begin{center}
\includegraphics{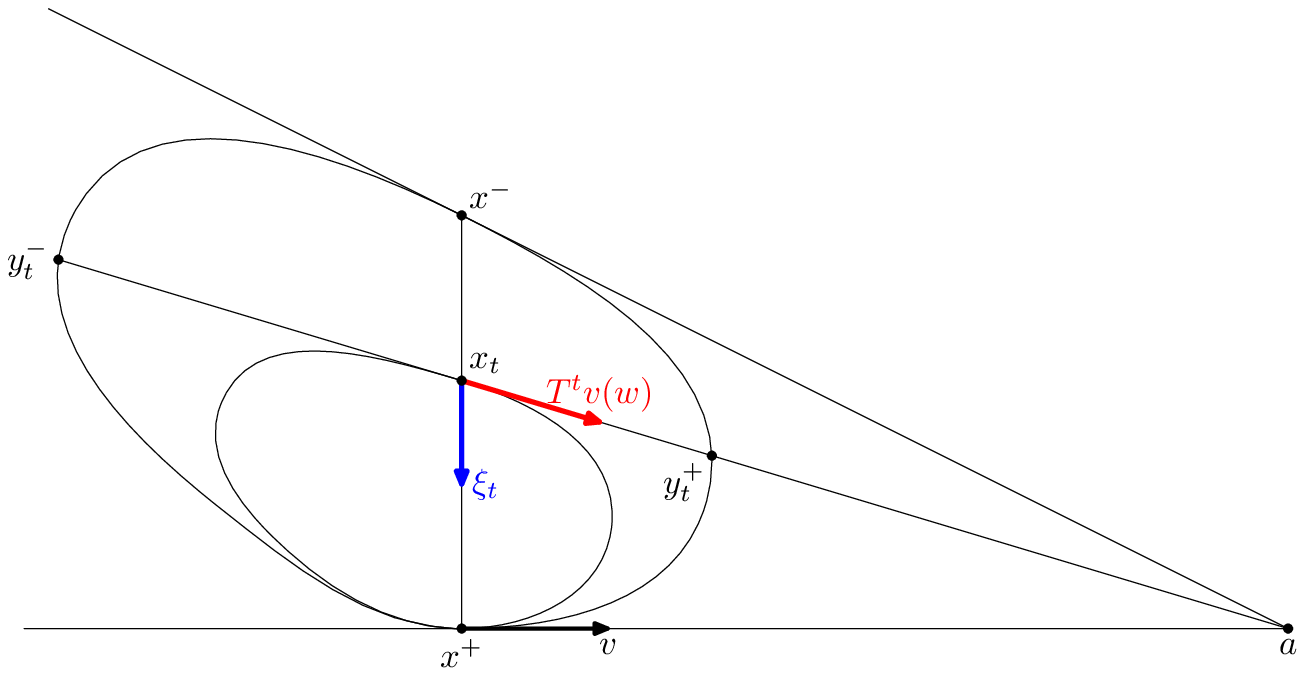}

\end{center}

\caption{For proposition \ref{linkboundary}}

\end{figure}

\begin{proof}
Let $w=(x,[\xi])$ be a point ending at $x^+$, $(x_t,[\xi_t]) = \ph^t(x,[\xi])$ its image by $\ph^t$, and $v\in T_x\H_w$. The vector $ T_{x^+}^t v$ is at any time contained in the plane generated by $\xi$ and $v$, thus, by working in restriction to this plane, we can assume that $n=2$. \\
We cannot choose a good chart at $w$, since the chart is already fixed. But, by affine invariance, we can choose the Euclidean metric $|\ .\ |$ and $\xi_t$ so that $\xi \bot T_{x^+}\doo = \R.p_{x^+}(v)$ and $|v|=|\xi_t|=1$. Let $a$ be the point of intersection of $T_{x^+}\doo$ and $T_{x^-}\doo$. The vector $ T_{x^+} v$ always points to $a$, that is, $ T_{x^+}^t v \in \R.x_ta$. Thus,

$$F( T_{x^+}^t v) = \frac{| T_{x^+}^t v|}{2} \left(\frac{1}{|x_ty_t^+|}+\frac{1}{|x_ty_t^-|}\right),$$
where $y_t^+$ and $y_t^-$ are the intersection points of $(ax_t)$ and $\doo$. If $f:U \subset T_{x^+}\doo \longrightarrow~\R$ denotes the function whose graph is a neighbourhood of $x^+$ in $\doo$, then
$$\frac{1}{2}\left(\frac{1}{|x_ty_t^+|}+\frac{1}{|x_ty_t^-|}\right) = \frac{1}{H(f_a^+,f_a^-)(|x_tx^+|)},$$
where $f_a^+$ and $f_a^-$ are defined as in corollary \ref{corotordu}. This corollary tells us that
$$\begin{array}{rl}

\displaystyle\limsup_{t\to+\infty} \frac{1}{t}\log\frac{1}{H(f_a^+,f_a^-)(|x_tx^+|)} & = \displaystyle\limsup_{t\to+\infty} \displaystyle-\frac{\log|x_tx^+|}{t}\displaystyle \frac{\log H(f_a^+,f_a^-)(|x_tx^+|)}{\log|x_tx^+|}\\\\

& = \displaystyle\limsup_{t\to+\infty} \displaystyle-\frac{\log|x_tx^+|}{t} \displaystyle\limsup_{s\to 0} \displaystyle\frac{\log H(f_a^+,f_a^-)(s)}{\log s} \\\\

& = \displaystyle\frac{2}{\underline{\alpha}(x^+,v)}

  \end{array}$$

(recall from lemma \ref{equivalents} that $|x_tx^+|= \frac{|xx^+|^2}{m(w)} e^{-2t}$). Hence
$$\limsup_{t\to+\infty} \frac{1}{t} \log F( T_{x^+}^t v) = \frac{2}{\underline{\alpha}(x^+,v)} + \limsup_{t\to+\infty} \frac{1}{t} \log | T_{x^+}^t v|.$$
From our choice of Euclidean metric, we have $| T^t v(w)| \asymp \langle  T^t v(w), v \rangle$. Lemma $\ref{horver}$ gives
$$ T_{x^+}^t v = -L_Y m(\ph^t w) \xi_t + (m(w)m(\ph^t w))^{1/2} d\pi(J^{X^e}(Y)),$$
where $Y\in VH\o$ is such that $d\pi(J^X(Y))=v(w)$; $d\pi(J^{X^e}(Y))$ is collinear to $v$ and has constant Euclidean norm, which implies  that
$$\lim_{t\to +\infty} \frac{1}{t}\log \langle  T_{x^+}^t v, v \rangle = \lim_{t\to +\infty} \frac{1}{t}\log (m(w)m(\ph^t w))^{1/2} = -1.$$

Hence
$$\overline{\eta}_{+}(w,v(w))=\limsup_{t\to+\infty} \frac{1}{t}F( T_{x^+}^t v) =   \frac{2}{\underline{\alpha}(x^+,v)} -1.$$

Obviously, the same holds for lower exponents.

\end{proof}

The last theorem tells us that the notions of Lyapunov regularity and exponents are projectively invariant, that is, it makes sense for codimension 1 submanifolds of $\R\P^n$. It then justifies the following

\begin{defi}
A locally strictly convex $\C^1$ submanifold $N$ of $\R\P^n$ is said to be \emph{approximately regular} at $x\in N$ if its trace in some (or, equivalently, any) affine chart at $x$ is locally the graph of an approximately regular function.
The numbers $\alpha_1(x)\geqslant\cdots\geqslant \alpha_p(x)$ attached to $x$ are called the \emph{Lyapunov exponents} of $x$.
\end{defi}

Also, remark the following properties:

\begin{corollary}\label{alpha}
Let  $f\in\mathtt{Cvx}(n)$. Then

\begin{itemize}
 \item the numbers $\overline{\alpha}(v),\ v\in \R^n\smallsetminus\{0\}$, can take only a finite numbers of values.  More precisely, there exist a number $\overline p$, a filtration
$$\{0\} = \overline{G}_0 \varsubsetneq \overline{G}_1 \varsubsetneq \cdots \varsubsetneq \overline{G}_{\overline p}=\R^n$$
and numbers
$$+\infty \geqslant \overline{\alpha}_1 > \cdots > \overline{\alpha}_{\overline p} \geqslant 1,$$
such that for any $v_i\in \overline{G}_i\smallsetminus \overline{G}_{i-1}$, $1\leqslant i \leqslant \overline p$,
$$\limsup_{t\to 0} \frac{\log \displaystyle\frac{f(tv_i) + f(-tv_i)}{2}}{\log |t|} = \overline{\alpha}_i.$$
The same holds for lower Lyapunov exponents.\\

\item the following propositions are equivalent:
\begin{enumerate}
 \item[(i)] $f$ is approximately regular;
 \item[(ii)] there exist a decomposition $\R^n=\oplus_{i=1}^p H_i$ and numbers $+\infty \geqslant \alpha_1 > \cdots > \alpha_p \geqslant 1$ such that the restriction $f|_{H_i\cap U}$ is approximately regular with exponent $\alpha_i$;
 \item[(iii)] there exist a filtration
$$\{0\}=G_0 \varsubsetneq G_1 \varsubsetneq \cdots \varsubsetneq G_p = \R^n$$ and numbers  $+\infty \geqslant \alpha_1 > \cdots > \alpha_p \geqslant 1$   such that, for any $v_i\in G_i\smallsetminus G_{i-1}$, the restriction $f|_{\R.v_i\cap U}$ is approximately regular with exponent $\alpha_i$.
\end{enumerate}
When $f$ is approximately regular, we call the numbers $\alpha_i$ the Lyapunov exponents of $f$.
\end{itemize}
\end{corollary}
\begin{proof}
The graph of $f$ can always be considered as the boundary of a strictly convex set $\o\subset \R^{n+1}$ with $\C^1$ boundary. We can then apply theorem \ref{linkboundary} to this set $\o$.
\end{proof}

\subsection{Lyapunov regularity of the boundary}

To characterize regular points $w\in H\o$, we need to add a property to approximate regularity because of the second point in definition \ref{regular}.

\begin{defi}\label{lyapunovregularity}
A function $f\in\mathtt{Cvx}(n)$ is said to be \emph{Lyapunov regular} if 
\begin{itemize}
\item $f$ is approximately regular with exponents $+\infty \geqslant \alpha_1 \geqslant \cdots \geqslant \alpha_p \geqslant 1$ counted with multiplicities;
\item $$\lim_{t\to 0} \frac{\log \displaystyle\int_{f(u)\leqslant |t|} |u|\ du}{\log t} = \frac{1}{\alpha} ,$$
where 
$$\frac{1}{\alpha} = \sum_{i=1}^n \frac{1}{\alpha_i}.$$
\end{itemize}
\end{defi}

\begin{defi}
A locally strictly convex $\C^1$ submanifold $N$ of $\R\P^n$ is said to be \emph{Lyapunov regular} at $x\in N$ if its trace in some (or, equivalently, any) affine chart at $x$ is locally the graph of a Lyapunov regular function.
\end{defi}

Remark that we should prove the second point in definition \ref{lyapunovregularity} is projectively invariant to state the last definition. In fact, we could proceed as before in theorem \ref{linkboundary} by proving the next theorem in any affine chart; but the idea is totally similar so we will not do it.

\begin{thm}\label{mainthm}
A point $w=(x,[\xi])\in H\o$ is forward regular if and only if the boundary $\doo$ is Lyapunov regular at the endpoint $x^+=\ph^{+\infty}(w)$.
The Lyapunov decomposition of $T\H_{x^+}$ along $\ph_{x^+}.x$ projects under $p_{x^+}$ on the Lyapunov decomposition of $T_{x^+}\doo$, and Lyapunov exponents are related by
$$\eta(v) = \frac{2}{\alpha(p_{x^+}(v))}-1,\ v\in T_x\H_w.$$
\end{thm}
\begin{proof}
The only if part is now clear from the last theorem. Assume $\doo$ is approximately regular at $x^+$. The decomposition of $T_{x^+}\doo$ gives by projection a decomposition
\begin{equation}\label{rrr}
 T_x\H_w = E_1 \oplus \cdots \oplus E_p, \end{equation}
such that, for any $v_i\in E_i\smallsetminus\{0\}$,
$$\lim_{t\to +\infty} \frac{1}{t} \log F(T_{x^+}^t v_i) = \eta_i,$$
where $\eta_1<\cdots <\eta_p$ are the parallel Lyapunov exponents of $w$. The only thing that we have to prove is the second point in definition \ref{regular}, that is,
$$\lim_{t\to +\infty} \frac{1}{t} \log \det T_{x^+}^t  = \sum_{i=1}^p \dim E_i\ \eta_i.$$
We can assume we have chosen a good chart and the Euclidean metric so that the decomposition \ref{rrr} is orthogonal. Recall that, by definition of the determinant and the Busemann volume,
$$\det T_{x^+}^t = vol T_{x^+}^t(B_{x}(1)) = \frac{vol^e (T_{x^+}^t(B_{x}(1)))}{vol^e (B_{\ph^t_{x^+}(x)}(1))}.$$
Since the map $T_{x^+}^t$ is linear, the quantity $vol^e (T_{x^+}^t(B_{x}(1)))$ is just the determinant $\det^e T_{x^+}^t $ of $T_{x^+}^t$ with respect to the Euclidean metric that we have chosen; lemma \ref{transport} implies that
$${\det}^e T_{x^+}^t  = (m(w)m(\ph^t(w))^{\frac{n-1}{2}} {\det}^e (T_{x^+}^e)^t,$$
so that
$$\lim_{t\to +\infty} \frac{1}{t} \log vol^e (T_{x^+}^t(B_{x}(1))) = \frac{n-1}{2} \lim_{t\to +\infty} \frac{1}{t} \log m(\ph^t(w)) = \frac{n-1}{2} \lim_{t\to +\infty} \frac{1}{t} \log |x_tx^+| = n-1,$$
by lemma \ref{equivalents}.\\
So we just have to study the quantity $ \frac{1}{t} \log vol^e (B_{\ph^t_{x^+}(x)}(1)).$
Call $x_t=\ph^t_{x^+}$ as usual, and for each vector $u\in T_x\H_w$, call $u_t$ the unit vector in $T_{\ph^t_{x^+}(x)}\o$ which is collinear to $u$. Since the vector $u_t$ has Finsler norm $1$, we have
$$1 = \frac{|u_t|}{2}\left(\frac{1}{|x_ty_t^+|}+\frac{1}{|x_ty_t^-|}\right),$$
so
$$|u_t| = \frac{2}{\frac{1}{|x_ty_t^+|}+\frac{1}{|x_ty_t^-|}} = m(x_t,[u_t]).$$
In particular, by lemmas \ref{transport} and \ref{equivalents},
$$\lim_{t\to+\infty} \frac{1}{t} \log |u_t| = -\eta(u) + 1.$$
By convexity of the unit balls, we then get
$$\lim_{t\to+\infty} \frac{1}{t} \log vol^e (B_{\ph^t_{x^+}(x)}(1)) \geqslant -\eta + (n-1) = -\frac{2}{\alpha}.$$
For the inequality from above, we just have to notice that
$$vol^e (B_{\ph^t_{x^+}(x)}(1)) \leqslant vol^e \left(\o \cap T_{\ph^t_{x^+}(x)}\H_{x^+}(\ph^t_{x^+}(x))\right),$$
hence
$$\begin{array}{rl}
\displaystyle\lim_{t\to+\infty} \frac{1}{t} \log vol^e (B_{\ph^t_{x^+}(x)}(1)) &\leqslant \displaystyle\lim_{t\to+\infty} \frac{1}{t} \log vol^e \left(\o \cap T_{\ph^t_{x^+}(x)}\H_{x^+}(\ph^t_{x^+}(x))\right) \\\\
&\leqslant  \displaystyle\lim_{t\to+\infty} -2\displaystyle\frac{\log vol^e \left(\o \cap T_{\ph^t_{x^+}(x)}\H_{x^+}(\ph^t_{x^+}(x))\right)}{\log |x_tx^+|},
\end{array}$$
from lemma \ref{equivalents}. The second property in definition \ref{lyapunovregularity} implies
$$\lim_{t\to+\infty} \frac{1}{t} \log vol^e (B_{\ph^t_{x^+}(x)}(1)) \leqslant -\frac{2}{\alpha}.$$
That means that $\lim_{t\to+\infty} \frac{1}{t} \log vol^e (B_{\ph^t_{x^+}(x)}(1)) = -\frac{2}{\alpha}$ and finally,
$$\lim_{t\to+\infty} \frac{1}{t} \log \det T_{x^+}^t = \frac{2}{\alpha} + (n-1) = \eta.$$
\end{proof}

\begin{rmk}
In reality, I am not sure the second property in definition \ref{lyapunovregularity} is necessary. I thought at the beginning it could be deduced from convexity and the other properties but I did not manage to prove it.
\end{rmk}







\vspace{1cm}

\section{Lyapunov manifolds of the geodesic flow}\label{sectionlyapunovmanifolds}


From the very definition of the metric $\|\ . \ \|$ (by using remark \ref{rmkdistance}), we get the following corollary of theorem \ref{mainthm}. Obviously, we could give an equivalent statement for non-approximately regular points by using upper and lower exponents.

\begin{corollary}\label{maincorollary}
Let $\o$ be a strictly convex proper open subfset of $\R\P^n$ with $\C^1$ boundary and fix $o\in\o$. Assume $x^+\in\doo$ is approximately regular with exponents $+\infty \geqslant \alpha_1 > \cdots > \alpha_p \geqslant 1$ and filtration
$$\{0\}=H_0 \varsubsetneq H_1 \varsubsetneq \cdots \varsubsetneq H_p = T_{x^+}\doo.$$
Then the horosphere $\H$ about $x^+$ passing through $o$ admits a filtration
$$\{o\}=\H_0 \varsubsetneq \H_1 \varsubsetneq \cdots \varsubsetneq \H_p = \H,$$
given by $\H_i=\{x\in\H\cap (H_i\oplus \R.ox^+)\},\ 1\leqslant i \leqslant p,$ and such that
$$\mathcal{H}_i\smallsetminus \mathcal{H}_{i-1} = \{x \in \H\smallsetminus\{o\},\ \lim_{t\to +\infty} \frac{1}{t} \log \d(\ph_{x^+}^t(o), \ph_{x^+}^t(x)) = \chi^s_i\},$$
with $\chi^s_i = -2+\frac{2}{\alpha_i}$.
\end{corollary}

This allows us to define Lyapunov manifolds of the geodesic flow, that is, submanifolds tangent to the subspaces appearing in the Lyapunov filtration. In the classical theory of nonuniformly hyperbolic systems, the local existence of these manifolds is a nontrivial result traditionnally achieved with the help of Hadamard-Perron theorem.\\
Here these manifolds appear naturally from the decomposition of the boundary at the endpoint of the orbit we are looking at. This result can be seen as a consequence of the flatness of Hilbert geometries.

\begin{corollary}
Assume $\doo$ is approximately regular at the point $x^+$. Each point $w$ of $W^s(x^+)$ is forward regular with decomposition
$$TH\o = E^s_0 \oplus (\oplus_{i=1}^p E^s_i) \oplus E_{p+1}^s \oplus \R.X \oplus E_0^u \oplus (\oplus_{i=1}^p E^u_i) \oplus E^u_{p+1},$$
and Lyapunov exponents
$$-2=\chi^s_0 < \chi_1^s < \cdots <\chi^s_p < \chi^s_{p+1} = 0 = \chi^u_0  < \chi^u_1 <\cdots<\chi^u_p < \chi^u_{p+1} = 2.$$
For each $w_0=(o,[ox^+])\in W^s(x^+)$, the stable manifold $W^s(w_0)$ admits a filtration by
$$\{w_0\} \subset W^s_0(w_0) \varsubsetneq W^s_1(w_0) \varsubsetneq \cdots \varsubsetneq W^s_p(w_0) \subset W^s_{p+1}(w_0)=W^s(w_0),$$
with
$$W^s_i(w):= \{w=(x,[xx^+])\in W^s(w_0),\ x\in \mathcal{H}_i\} = \{w \in H\o,\ \limsup_{t\to +\infty} \frac{1}{t} \log d_{H\o}(\ph^t(w_0),  \ph^t(w)) \leqslant \chi^-_i\}.$$
The tangent distribution to $W^s_i(w)$ is precisely $\oplus_{k=0}^i E^s_k$. 
(Recall that the subspaces $E^s_0$ and $E_{p+1}^s$ can be $\{0\}$, in which case $W^s_0(w_0)=\{w\}$, and $W^s_{p}(w_0)=W^s_{p+1}(w_0)=W^s(w_0)$.)
\end{corollary}

Obviously, the last corollary can be stated also for an approximately regular point $x^-\in\doo$ and the corresponding unstable manifold $$W^u(x^-) = \{w\in H\o,\ \ph^{-\infty}(w)=x^-\}.$$

\subsection{Non-strict convexity, non-$C^1$ points}\label{extension}

We now explain how to extend corollary \ref{maincorollary} to an arbitrary convex set. Let $\o$ be \emph{any} convex proper open subset of $\R\P^n$ and choose a point $x^+\in\doo$. The flow $\ph_{x^+}^t$ is well defined, the definition of approximate regularity given in section \ref{sectionlyapunovboundary} still makes sense and the results we achieve before can be extended to this general convex set by using the following easy lemma.

\begin{lemma}
Let $\o$ be any proper convex subset of $\R\P^n$ and $x\in\doo$.
\begin{itemize}
 \item The maximal flat
$$\mathcal{F}(x)=\{y\in\doo,\ [xy]\subset\doo\}$$
containing $x$ in $\doo$ is a closed convex subset of a projective subspace $\R\P^q$, for some $0\leqslant q\leqslant n-1$, whose interior is open in this $\R\P^q$ when $\mathcal{F}(x)$ is not reduced to $\{x\}$.
 \item The set of $C^1$ directions
$$\mathcal{D}(x) = \{0\}\cup\{v\in T_x\doo\smallsetminus\{0\},\ \doo\ \text{is differentiable in the direction}\ v \}$$
is a subspace of $T_x\doo$.
\end{itemize}
\end{lemma}
\begin{proof}
The set $\mathcal{F}(x)$ is obviously closed. It is convex because of the convexity of $\o$. The projective subspace $\R\P^q$ is the one spanned by $\mathcal{F}(x)$. The second point is just a consequence of convexity.
\end{proof}

Choose a direction $v\in T_x\doo$ in which the boundary $\doo$ is not differentiable and any vector $u\not \in T_x\doo$. We can consider the $2$-dimensional convex set $C_{v}(w)=\o \cap (\R.v \oplus \R.u)$. As we have seen in the introduction, for two distinct geodesic lines of $C_v(u)$ ending at $x$, the distance between them does not tend to $0$. Hence the negative Lyapunov exponent $\chi^s$ of such a geodesic, if it were defined, would be $\chi^s=0$; it is coherent with the fact that $\alpha(v) = 1$ and the relation $\chi^s = -2 + \frac{2}{\alpha(v)}$. \\
We can now consider the subspace $\mathcal{D}(x)$ of $\C^1$ directions and the convex set $C_{x}(u) = \o \cap (D(x) \oplus \R.u)$ for an arbitrary vector $u\not \in T_x\doo$. For example, the stable manifold $\H^s_{x^+}(x)$ of $\ph_{x^+}^t$ at $x$ is the set
$$\H^s_{x^+}(x) = C_x(xx^+)\cap \H_{x^+}(x).$$
The boundary $\partial C_x(u)$ is $\C^1$ at Lebesgue-almost every point $x^-$, so all we did before is relevant along Lebesgue almost-all geodesic $(x^-x^+)$. We just have to be careful for those vectors in $\span\ \mathcal{F}(x)$ which were not considered before: in such a direction $v$, the boundary is obviously $+\infty$-approximately regular, and as we have seen in the introduction, the distance between two geodesics of $C_v(u)$ with origin on the same horosphere and ending at $x^+$ goes to $0$ as $e^{-2t}$.\\

As a consequence, we get that corollary \ref{maincorollary} is valid for any Hilbert geometry:

\begin{corollary}\label{maincorollary2}
Let $\o$ be a convex proper open subfset of $\R\P^n$ and fix $o\in\o$. Assume $x^+\in\doo$ is approximately regular with exponents $+\infty \geqslant \alpha_1 > \cdots > \alpha_p \geqslant 1$ and filtration
$$\{0\} = H_0 \varsubsetneq H_1 \varsubsetneq \cdots \varsubsetneq H_p = T_{x^+}\doo.$$
Then the horosphere $\H=\H_{x^+}(o)$ about $x^+$ passing through $o$ admits a filtration
$$\{o\}=\H_0 \varsubsetneq \H_1 \varsubsetneq \cdots \varsubsetneq \H_p = \H,$$
given by $\H_i=\{x\in\H\cap (H_i\oplus \R.ox^+)\},\ \leqslant i \leqslant p,$ such that
$$\mathcal{H}_i\smallsetminus \mathcal{H}_{i-1} = \{x \in \H\smallsetminus\{o\},\ \lim_{t\to +\infty} \frac{1}{t} \log \d(\ph_{x^+}^t(o), \ph_{x^+}^t(x)) = \chi^s_i\},$$
with $\chi^s_i = -2+\frac{2}{\alpha_i}$.
\end{corollary}

In this last corollary, if $\mathcal{F}(x^+)$ is not reduced to $x^+$, then the subspace $H_1$ itself admits a filtration $\{0\} \varsubsetneq \span\ \mathcal{F}(x^+) \varsubsetneq H_1$; $H_1\smallsetminus \span\ \mathcal{F}(x^+)$ consists of these vectors $v$ with Lyapunov exponent $\alpha(v) = +\infty$ which are not in $\span\ \mathcal{F}(x^+)$, that is, the directions in which $\doo$ is not flat, but infinitesimally flat. Of course this also provides a filtration of $\H_1$.\\
Similarly, if $\mathcal{D}(x^+)$ is not all of $T_{x^+}\doo$, we can refine the filtration into
$$\cdots \varsubsetneq H_{p-1}  \varsubsetneq  \mathcal{D}(x^+) \varsubsetneq H_p = T_{x^+}\doo.$$
The subspace $\mathcal{D}(x^+)$ is precisely the tangent space to the stable manifold $\H^s_{x^+}(o)$ of $\ph^t_{x^+}$ at $o$, and $\H$ admits a subfiltration
$$\cdots \varsubsetneq \H_{p-1}  \varsubsetneq  \H^s_{x^+}(o) \varsubsetneq  \H_p = \H.$$


\vspace{1cm}

\section{Examples}

I do not know what can be said in general about the notion of approximate-regularity for a given strictly convex set $\o$ with $\C^1$-boundary. We can relate this with Alexandrov's theorem which says that the boundary $\doo$ of $\o$ is $\C^2$ Lebesgue-almost everywhere. This implies that for almost every point $x\in\doo$, we have $\underline{\alpha}(v)\geqslant 2$ for all vectors $v\in T_x\doo$. It might be interesting for example to know if $\doo$ is approximately regular at almost every point.\\
Here I give some more properties of approximate-regularity and study the case of divisible convex sets. In particular I show that in this case $\doo$ is approximately regular at almost every point with the same Lyapunov exponents.

\subsection{Duality and approximate regularity}

\subsubsection{Legendre transform}\label{sectionlegendre}

Pick a function $f\in \mathtt{Cvx}(n)$. Since $f$ is $\C^1$ and strictly convex, the gradient
$$\nabla: x\in U \longmapsto \nabla_x f = \left(\frac{\partial f}{\partial x_1}(x),\cdots, \frac{\partial f}{\partial x_n}(x)\right)$$
is an injective map onto a convex subset $V$ of $\R^n$. Using the gradient, a point $x$ can thus be defined by its coordinates $(x_1,\cdots,x_n)$ or by its ``dual'' coordinates $\left(\frac{\partial f}{\partial x_1}(x),\cdots, \frac{\partial f}{\partial x_n}(x)\right)$.\\
The {\bf Legendre transform} of $f$ is the function $f^*$ defined by
$$f(x) + f^*(\nabla_x f) = \langle \nabla_x f, x \rangle.$$
It happens that the transform $f\longmapsto f^*$ is an involution of $\mathtt{Cvx}(n)$. We will see in the next section that it appears naturally when one considers the dual of a convex set. Our goal in the next section is to make a link between the shape of the boundary of the convex set and the one of its dual. For this, we study here the link between the approximate regularity of $f$ and of its Legendre transform $f^*$. I am not very familiar with Legendre transform and I did not manage to prove the next lemma in higher dimensions; but it is probably true...

\begin{lemma}\label{legendre}
Assume $f\in \mathtt{Cvx}(1)$ is approximately $\alpha$-regular, $\alpha\in[1,+\infty]$. Then the Legendre transform $f^*$ of $f$ is approximately $\alpha^*$-regular with
$$\frac{1}{\alpha^*} + \frac{1}{\alpha} = 1.$$
\end{lemma}
\begin{proof}
We only prove the proposition when $\alpha\in (1,+\infty)$. The Legendre transform of $f\in \mathtt{Cvx}(1)$ is given by
$$f^*(f'(x))=xf'(x) - f(x).$$
By considering $f(x)+f(-x)$ instead, we can assume that $f$ is an even function, so that approximate $\alpha$-regularity gives
$$\lim_{x\to 0^+} \frac{\log f(x)}{\log x} = \alpha.$$
Since $f(0)=f(x)-xf'(x) + o(x)$, we get
$$\lim_{x\to 0^+} \frac{\log f'(x)}{\log x} = \alpha-1.$$
We need to understand the limit
$$\lim_{x\to 0^+} \frac{\log xf'(x)-f(x)}{\log f'(x)}.$$
Fix $\epsilon>0$. There is some $x>0$ such that for $0\leqslant t\leqslant x$, we have
\begin{equation}\label{ggg}
 t^{\alpha-1+\epsilon} \leqslant f'(t) \leqslant t^{\alpha-1-\epsilon}.
\end{equation}
Remark that
$$xf'(x) - f(x) = xf'(x) - \int_0^x f'(t) dt.$$
From (\ref{ggg}), that means the value of $xf'(x) - f(x)$ is in between the two areas between $0$ and $x$ delimited by the line $y=f'(x)$ above and, respectively, the curves $t\mapsto t^{\alpha-1-\epsilon}$ and $t\mapsto t^{\alpha-1+\epsilon}$ below:
$$ f'(x)^{\frac{1}{\alpha-1-\epsilon}}f'(x) - \int_0^{f'(x)^{\frac{1}{\alpha-1-\epsilon}}} t^{\alpha-1-\epsilon}\ dt \leqslant xf'(x) - f(x) \leqslant xf'(x) - \int_0^x t^{\alpha-1+\epsilon}\ dt.$$
Hence
$$(f'(x))^{\frac{\alpha-\epsilon}{\alpha-1-\epsilon}} - \frac{1}{\alpha-\epsilon} (f'(x))^{\frac{\alpha-\epsilon}{\alpha-1-\epsilon}} \leqslant xf'(x) - f(x) \leqslant xf'(x) - \frac{1}{\alpha+\epsilon} x^{\alpha+\epsilon}.$$
Using (\ref{ggg}) again, we get
$$\frac{\alpha-\epsilon-1}{\alpha-\epsilon}(f'(x))^{\frac{\alpha-\epsilon}{\alpha-1-\epsilon}} \leqslant xf'(x) - f(x) \leqslant x^{\alpha-\epsilon} - \frac{1}{\alpha+\epsilon} x^{\alpha+\epsilon} \leqslant x^{\alpha-\epsilon}.$$
So
$$ \frac{\alpha-\epsilon}{\alpha-1} = (\alpha -\epsilon)\lim_{x\to 0^+} \frac{\log x}{\log f'(x)}\leqslant \lim_{x\to 0^+} \frac{\log xf'(x)-f(x)}{\log f'(x)} \leqslant \frac{\alpha-\epsilon}{\alpha-1-\epsilon}.$$
Since $\epsilon$ is arbitrary small, we get the result.
\end{proof}

\subsubsection{Dual convex set}

To each convex set $\o\subset\R\P^n$ is associated its dual convex set $\o^*$. To define it, consider one of the two convex cones $C \subset \R^{n+1}$ whose trace is $\o$. The dual convex set $\o^*$ is the trace of the dual cone
$$C^* = \{f\in (\R^{n+1})^*,\ \forall x\in C,\ f(x)> 0\}.$$
The cone $C^*$ is a subset of the dual of $\R^{n+1}$ but of course, it can be seen as the subset
$$\{y\in \R^{n+1},\ \forall x\in C,\ \langle x,y \rangle \geqslant 0\}.$$

The set $\o^*$ can be identified with the set of projective hyperplanes which do not intersect $\overline\o$: to such a hyperplane corresponds the line of linear maps whose kernel is the given hyperplane. For example, we can see the boundary of $\o^*$ as the set of tangent spaces to $\doo$. In particular, when $\o$ is strictly convex with $\C^1$ boundary, there is a homeomorphism between the boundaries of $\o$ and $\o^*$: to the point $x\in\doo$ we associate the (projective class of the) linear map $x^*$ such that $\ker x^* = T_x\doo$.\\

In the following we would like to link the shape of $\doo$ and $\doo^*$. We will work in $\R^{n+1}$ with the cones $C$ and $C^*$ where it is more usual to make computations. Choose a point $p\in\partial C$ and fix a Euclidean structure on $\R^{n+1}$ and an orthonormal basis $(u_1,\cdots,u_{n+1})$ so that $p=u_1+u_{n+1}$, $T_p\partial C = \span \{p,u_2,\cdots,u_{n+1}\}$ and $C \subset \{x=(x_1,\cdots,x_{n+1}),\ x_{n+1}>0\}$. We identify $\o$ with the intersection $C\cap\{x_{n+1}=1\}$ and the tangent space $T_p\doo$ is $p+\span \{u_2,\cdots,u_{n+1}\}$.\\
Call $f : U\subset T_p\doo \longrightarrow \R$ the local graph of $\doo$ at $p$, such that, around $p$,
$$\doo = \{(1-f(x_2,\cdots,x_n),x_2,\cdots,x_n,1)\}.$$

\begin{lemma}
Around $p^*=(1,0,\cdots,0,-1)$, the boundary $\doo^*$ is given by
$$\doo^* = \{(1,\l_2,\cdots,\l_n,-1-f^*(\l_2,\cdots,\l_n))\},$$
where $f^*$ is the Legendre transform of $f$. In other words, the local graph of $\doo^*$ at $p$ is given by the Legendre tranform $f^*$ of $f$.
\end{lemma}
\begin{proof}
Take a point $x = (1-f(x_2,\cdots,x_n),x_2,\cdots,x_n,1)\in\doo$ and call $x_{2n}=x_2u_2+\cdots+x_nu_n\in \span \{u_2,\cdots,u_{n+1}\}$ its projection on $\span \{u_2,\cdots,u_{n+1}\}$. Call $F : T_p\doo \longrightarrow \R^{n+1}$ the map given by $$F(p+x_{2n}) = p + x_{2n} - f(x_{2n})u_1 = (1-f(x_2,\cdots,x_n),x_2,\cdots,x_n,1).$$
The tangent space of $\doo$ at $x$ is then given by
$$T_x\doo = x + d_xF(\span \{u_2,\cdots,u_{n+1}\}).$$
But, for $h\in\span \{u_2,\cdots,u_{n+1}\}$, we have $d_xF(h) = -d_xf(h) + h$. Hence
$$T_x\doo = x + \{ h - d_xf(h),\ h\in \span \{u_2,\cdots,u_{n+1}\} \}.$$
Now the dual point of $x$ is the linear map $x^*=(x_1^*,\cdots,x_{n+1}^*)$ such that $x^*(x)=0$, $x^*(T_x\doo)=0$ and $x^*(u_1)=1$. (This last condition is just a normalization condition, since there is a line of corresponding linear maps.) The third condition gives $x_1^* = 1$. The second implies that for any $h\in\span \{u_2,\cdots,u_{n+1}\}$,
$$0 = x^*(h - d_xf(h))=\langle x^* - \nabla_x f, h \rangle;$$
hence $x_{2n}^* = \nabla_x f$, that is, $x_i^* = \frac{\partial f}{\partial u_i}(x_{2n}),\ i=2,\cdots, n$. Finally, the first condition gives
$$1-f(x_2,\cdots,x_n) + \langle \nabla_x f, x_{2n}\rangle + x_{n+1}^* = 0,$$
so
$$x_{n+1}^* = -1-(\langle \nabla_x f, x_{2n}\rangle - f(x_{2n})).$$
By considering the set of variables $(\l_2,\cdots,\l_n)= \left(\frac{\partial f}{\partial u_2},\cdots, \frac{\partial f}{\partial u_n}\right)$, one finally gets
$$x^* = (1,\l_2,\cdots,\l_n,-1-f^*(\l_2,\cdots,\l_n)).$$
\end{proof}

From lemma \ref{legendre}, we get the following

\begin{corollary}
Assume $\doo\subset\R\P^2$ is approximately $\alpha$-regular at the point $x$. Then $\doo^*$ is approximately $\alpha^*$-regular at the point $x^*$ with
$$\frac{1}{\alpha^*} + \frac{1}{\alpha} = 1.$$
\end{corollary}

\subsection{Hyperbolic isometries}

If $\o$ is strictly convex with $\C^1$ boundary, the group of isometries $Isom(\o,\d)$ of the Hilbert geometry $(\o,\d)$ consists of those projective transformations which preserve the convex set $\o$:
$$Isom(\o,\d) = \{g\in PGL(n+1,\R),\ g(\o)=\o\}.$$
As in the hyperbolic space, isometries can be classified into three types, elliptic, parabolic and hyperbolic. This is proved in the forthcoming paper \cite{cramponmarquis}.\\
A hyperbolic isometry $g$ fixes exactly two points $x_g^+$ and $x_g^-$ on $\doo$. The point $x_g^+$ is the attractive point of $g$, $x_g^-$ is the repulsive point of $g$ : for any point $x\in\overline{\o}\smallsetminus\{x_g^-,x_g^+\}$, $\lim_{n\to \pm\infty} g^n(x)=x_g^{\pm}$. These two points are the eigenvectors associated to the biggest and smallest eigenvalues $\l_0$ and $\l_{p+1}$ of $g$. The isometry $g$ acts as a translation of length $\log \frac{\l_{p+1}}{\l_0}$ on the open segment $]x_g^-x_g^+[$. The following result is proved in \cite{crampon}:

\begin{prop}\label{lyapunovperiodic}
Let $g$ be a periodic orbit of the flow, corresponding to a hyperbolic element $g\in\Gamma$. Denote by $\l_0 > \l_1 > \cdots > \l_{p} > \l_{p+1}$ the moduli of the eigenvalues of $g$. Then
\begin{itemize}
 \item $\g$ is regular and has no zero Lyapunov exponent;
 \item the  Lyapunov exponents $(\eta_i(g))$ of the parallel transport along $\g$ are given by
$$\eta_i(g) = -1 + 2\ \frac{\log \l_0-\log \l_i}{\log \l_{0} -\log \l_{p+1}},\ i=1\cdots p;$$
\item the sum of the parallel Lyapunov exponents is given by
$$\eta(g) = (n+1)\frac{\log \l_0 +\log \l_{p+1}}{\log \l_0 - \log \l_{p+1}}.$$
\end{itemize}
\end{prop}

As a consequence of the results before, we see that, if $g$ is a hyperbolic isometry, the boundary $\doo$ is Lyapunov regular at the points $x_g^-$ and $x_g^+$, with Lyapunov exponents
\begin{equation}\label{exposanthyp}
\alpha_i = \frac{\log \l_{0} -\log \l_{p+1}}{\log \l_0-\log \l_i},\ i=1\cdots p.
\end{equation}

The isometry $g\in Isom(\o,\d)$ acts on the dual convex set $\o^*$ by $g.y = (^tg)^{-1}(y)$. To $g\in Isom(\o,\d)$, we thus associate the isometry $g^* =  (^tg)^{-1}\in Isom(\o^*,d_{\o^*})$. The dual points to $x_g^-$ and $x_g^+$ are respectively the points $x_{g^*}^+$ and $x_{g^*}^-$, at which $\doo^*$ is Lyapunov regular with Lyapunov exponents
$$\alpha_i^* = \frac{\log \l_{0} -\log \l_{p+1}}{\log \l_i-\log \l_{p+1}},\ i=1\cdots p:$$
this corresponds to what gives formula (\ref{exposanthyp}) for the isometry $g^{-1}$. Remark that, as expected, we have
$$\frac{1}{\alpha_i^*} + \frac{1}{\alpha_i} = 1,\ i=1\cdots p.$$

\subsection{Divisible convex sets}

The convex set $\o$ is said to be divisible if it admits a discrete cocompact subgroup $\G$ of projective isometries. By Selberg lemma, we can assume $\G$ has no torsion and the quotient $M=\o/\G$ is then a smooth manifold. The first example of divisible convex set is the ellipsoid, that is, the hyperbolic space. Benoist proved in \cite{benoistcv1} that, for a divisible convex set $\o$, the following properties were equivalent:\\

\begin{itemize}
 \item $\o$ is strictly convex;
 \item $\doo$ is of class $\C^1$;
 \item $(\o,\d)$ is Gromov-hyperbolic.\\
\end{itemize}
Apart from the ellipsoid, various examples of strictly convex divisible sets have been given. Some can be constructed using Coxeter groups (\cite{kav}, \cite{benoistqi}), some by deformations of hyperbolic manifolds (based on \cite{johnsonmillson} and \cite{koszul}, see also \cite{goldman} for the $2$-dimensional case); we should also quote the exotic examples of Kapovich \cite{kapo} of divisible convex sets in all dimensions which are not quasi-isometric to the hyperbolic space (Benoist \cite{benoistqi} had already given an example in dimension 4).\\

In what follows, we are given a compact manifold $M=\o/\G$, quotient of a strictly convex set $\o$ with $\C^1$ boundary.\\

\subsubsection{Regularity of the boundary}

Benoist proved that the geodesic flow on $HM$ has the Anosov property, with decomposition 
$$THM = \R.X \oplus E^u \oplus E^s.$$
That means there exist constants $C,\alpha>0$ such that for any $ t\geqslant 0$,
$$ \|d\ph^t(Z^s)\| \leqslant C e^{-\alpha t} \|Z^s\|,\ Z^s\in E^s,$$
$$\|d\ph^{-t}(Z^u)\| \leqslant C e^{-\alpha t} \|Z^u\|,\ Z^u \in E^u.$$

As a consequence, we get that the boundary $\doo$ is $\C^{\alpha}$ and $\beta$-convex for some $1<\alpha\leqslant 2 \leqslant \beta<+\infty$. This had already been remarked by Benoist in \cite{benoistcv1}, and Guichard proved that the biggest $1<\alpha \leqslant 2$ and smallest $2 \leqslant\beta$ one can take are related to the group $\G$:
$$\alpha(\o) = \sup_{g\in\G} \frac{\log \l_{0}(g) -\log \l_{p+1}(g)}{\log \l_0(g)-\log \l_1(g)}.$$
$$\beta(\o) = \inf_{g\in\G} \frac{\log \l_{0}(g) -\log \l_{p+1}(g)}{\log \l_0(g)-\log \l_p(g)}.$$
Guichard result is stated in another form: the dual group $\G^*$ also acts cocompactly on the dual convex set $\o^*$, providing another compact manifold $M^*=\o^*/\G^*$; Guichard showed that $\alpha(\o)=\alpha(\o^*)$ and $\beta(\o)=\beta(\o^*)$. In \cite{cramponmarquis}, we will give another proof of Guichard result that we also extend to some non cocompact actions.\\

The case of the ellipsoid is a particular one. Indeed, the following facts are equivalent:
\begin{itemize}
\item $\o$ is an ellipsoid;
\item $\alpha(\o) = \beta(\o) = 2$;
\item $\G$ is not Zariski-dense in $SL(n+1,\R)$;
\item the parallel transport on $HM$ is an isometry;
\item the Lyapunov exponents are to $-1$, $0$ and $1$, corresponding to the Anosov decomposition $THM = E^s\oplus \R.X \oplus E^u$. 
\end{itemize}

\subsubsection{Ergodic measures}

Let $\Lambda(H\o)$ be the set of regular points on $H\o$, which is obviously $\G$-invariant, and call $\Lambda$ the projection of $\Lambda(H\o)$ on $HM$. From Oseledets' theorem, we know that for any invariant measure $m$ of the geodesic flow on $HM$, $\Lambda$ has full $m$-measure; in particular, Lyapunov exponents are defined almost everywhere. If $m$ is an ergodic measure, that is, such that invariant sets have zero or full measure, then Lyapunov exponents are constant almost everywhere: to each ergodic measure $m$ we can thus associate a number $p=p(m)$ and its parallel Lyapunov exponents $\eta_1(m) < \cdots < \eta_p(m)$.\\
Kaimanovich \cite{kaimanovich} explained how to associate in a one-to-one way to each invariant probability measure $m$ on $HM$ a $\G$-invariant Radon measure $M=M(m)$ on the space of oriented geodesics of $\o$ given by $\partial^2\o=\doo\times\doo\smallsetminus\Delta$, where $\Delta = \{(x,x),\ x\in\doo\}$. If $m$ is ergodic, Oseledets' theorem implies that for $M$-almost all $(x,y)\in\partial^2\o$, the geodesic from $x$ to $y$ is regular with parallel Lyapunov exponents $\eta_1(m) < \cdots < \eta_p(m)$; thus, for $M$-almost all $(x,y)\in\partial^2\o$, the boundary $\doo$ is Lyapunov regular at $x$ and $y$ with Lyapunov exponents $\alpha_i(m),\ 1\leqslant i \leqslant p$, given by
$$\alpha_i(m) = \frac{2}{\eta_i(m)+1}.$$
By projecting on the first and second coordinates in $\partial^2\o$, we get for each ergodic measure $m$ two $\G$-invariant sets $\doo^-(m)$ and $\doo^+(m)$ where the boundary $\doo$ is Lyapunov regular with the same Lyapunov exponents $\alpha_i(m),\ 1\leqslant i \leqslant p$. Recall that the action of $\G$ on $\doo$ is minimal, that is, every orbit is dense; the sets $\doo^-(m)$ and $\doo^+(m)$ are then dense subsets of $\doo$.\\
The diversity of invariant measures can then give an idea of the complexity of the boundary of a divisible convex set. Here are some examples.\\

The easiest examples of ergodic measures are the Lebesgue measures $l_{g}$ supported by a closed orbit $g$, associated to a conjugacy class of a hyperbolic element $g\in \G$. The corresponding set of full $M(l_{g})$-measure is precisely the orbit of $(x_g^-,x_g^+)$ under $\G$ while its projections $\doo^-(m)$ and $\doo^+(m)$ are the $\G$-orbits of $x_g^-$ and $x_g^+$.\\

Other examples are provided by Gibbs measure which are equilibrium states of H\"older continuous potentials $f:HM \longrightarrow \R$: the Gibbs measure of $f$ is the unique invariant probability measure $\mu_f$ such that
$$h_{\mu_f} + \int f\ d\mu_f = \sup\{h_{m} + \int f\ dm,\ m\ \text{invariant probability measure}\}.$$ 
Two distinct potentials $f$ and $g$ have the same equilibrium states if and only if their difference is invariant under the flow. The corresponding measure $M_f$ on $\partial^2\o$ can always be written as $M_f = F M_f^s\times M_f^u$, where $F$ is a continuous function on $\partial^2\o$, and $M_f^s$ and $M_f^u$ are two finite measures on $\doo$. The three objects are determined by the potential; in particular, $M_f^u$ and $M_f^s$ are given by the Patterson-Sullivan construction, associated to the potentials $f$ and $\sigma * f$, where $\sigma$ is the flip map.\\
Among them are two particular measures. The first one is the Bowen-Margulis measure $\mu_{BM}$ which is the measure of maximal entropy of the flow, that is, the equilibrium state associated to the potential $f\equiv 0$. The corresponding measure $M_{BM}$ is given by
$$dM_{BM}(\xi^+,\xi^-) =e^{2\delta(\xi^+|\xi^-)_o} d\mu_{o}^2(\xi^+,\xi^-),$$
where $\mu_{o}$ is the Patterson-Sullivan measure at an arbitrary point $o\in\o$, and $(\xi^+|\xi^-)_o$ is the Gromov product $\xi^+$ and $\xi^-$ based at the point $o$. In \cite{crampon}, I had proved that $\eta(\mu_{BM}) = \sum \eta_i(\mu_{BM}) = n-1$. Thus, we get that $\mu_o$-almost every point of $\doo$ is Lyapunov regular with exponents $\alpha_i,\ i=1,\cdots,p$, such that $\alpha = 2(n-1)$, with
$\frac{1}{\alpha} = \sum_i \frac{1}{\alpha_i}$. For example, in dimension $2$,  $\mu_o$-almost every point of $\doo$ is Lyapunov $2$-regular. A question I could not answer was to know if, in dimension $n\geqslant3$, there was only one parallel Lyapunov exponent if and only if $\o$ was an ellipsoid, that is, $M$ was a hyperbolic manifold.\\

The second measure which is important is the Sinai-Ruelle-Bowen (SRB) measure $\mu^+$, which is the equilibrium state associated to the potential
$$f^+ = \frac{d}{dt}|_{t=0} \log \det d\ph^t_{|_{E^u}}.$$
It is the only invariant measure whose conditional measures $(\mu^+)^u$ along unstable manifolds are absolutely continuous, and which satisfies the equality in the Ruelle inequality. Recall that the Ruelle inequality relates the entropy of an invariant measure $m$ to the sum of positive Lyapunov exponents $\chi^+=n-1+\eta$ of the flow:
$$h_{m} \leqslant \int \chi^+\ dm.$$
Closely related to this measure is the ``reverse'' SRB measure $\mu^- = \sigma* \mu^+$, which is the equilibrium state of the potential
$$f^- = \frac{d}{dt}|_{t=0} \log \det d\ph^t_{|_{E^s}}.$$
The measure $\mu^-$ is the only invariant measure whose conditional measures along stable manifolds are absolutely continuous.\\
In the case of the ellipsoid, $\mu^+$, $\mu^-$ and $\mu_{BM}$ all coincide, since $f^+=f^-=0$, and they are all absolutely continuous; indeed, they coincide with the Liouville measure of the flow. When $\o$ is not an ellipsoid, the Zariski-density of the cocompact group $\G$ implies via Livschitz-Sinai theorem that there is no absolutely continuous measure (see \cite{benoistcv1}). So the three measures are distinct.\\
The measures $\mu^+$ and $\mu^-$ have the same entropy $h_{SRB}$ given by
$$h_{SRB} = \int \chi^+\ d\mu^+ = -\int \chi^-\ d\mu^-,$$
where $\chi^- = -(n-1)+\eta$ is the sum of negative Lyapunov exponents. In particular, since the Bowen-Margulis measure is the measure of maximal entropy and has entropy $h_{BM}\leqslant n-1$, from Ruelle inequality, we get that the almost sure value $\eta(SRB)$ (with respect to $\mu^+$ or $\mu^-$) of the sum of parallel Lyapunov exponents satisfies $\eta(SRB) < 0.$\\
The measure $\mu^+$ corresponds to the measure $M^+$ on $\partial^2\o$ which can be written $M^+ = F^+ M^s \times M^u$, with $M^u$ absolutely continuous, while the measure $\mu^-$ corresponds $M^- = F^- M^u \times M^s$. In particular,

\begin{corollary}
Let $\o$ be a divisible strictly convex set. Then Lebesgue-almost every point of $\doo$ is Lyapunov regular with exponents $$\alpha_i(SRB) = \frac{2}{\eta_i(SRB)+1},\ 1\leqslant i \leqslant p.$$
\end{corollary} 

Since $\doo$ is also $\C^2$ Lebesgue almost-everywhere, we have that $\alpha_i(SRB)\leqslant 2$. When $\o$ is an ellipsoid, we have $p=1$ and $\alpha_1(SRB)=2$. In the other cases, the fact that $\eta(SRB)<0$ implies that $\eta_1(SRB)<0$ hence $\alpha_1(SRB)>2$. In particular, we recover the fact that the curvature of $\doo$ is concentrated on a set of Lebesgue-measure $0$ (see \cite{benoistcv1}).


\vspace{1cm}

\section{About volume entropy}\label{sectionvolentropy}

The volume entropy of a Riemannian metric $g$ on a manifold $M$ measures the asymptotic exponential growth of the volume of balls in the universal cover $\tilde M$; it is defined by
\begin{equation}\label{entvol} h_{vol}(g)=\limsup_{R\to+\infty} \frac{1}{r} \log vol_g(B(x,R)), \end{equation}

where $vol_g$ denotes the Riemannian volume corresponding to $g$. We define the volume entropy of a Hilbert geometry $(\o,\d)$ by the same formula, with respect the Busemann volume.\\

Some results are already known: for instance, if $\o$ is a polytope then $h_{vol}(\o,d_{\o})=0$; at the opposite, we have the
\begin{thm}[\cite{bbv}]\label{volentropy}
Let $\o\subset\R\P^n$ be a convex proper open set. If the boundary $\partial\o$ of $\o$ is $\C^{1,1}$, that is, has Lipschitz derivative, then $h_{vol}(\o,d_{\o})=n-1$.
\end{thm}
The global feeling is that any Hilbert geometry is in between the two extremal cases of the ellipsoid and the simplex. In particular, the following conjecture is still open:
\begin{conjecture}\label{conjectureentropy}
For any $\o\subset \R\P^n$,
$$h_{vol}(\o,d_{\o})\leqslant n-1.$$
\end{conjecture}

In \cite{bbv} the conjecture is proved in dimension $n=2$ and an example is explicitly constructed where $0<h_{vol}<1$. Following their idea for proving theorem \ref{volentropy}, we can get the

\begin{prop}
Let $(\o,\d)$ be any Hilbert geometry, and $\mathcal{L}$ a probability Lebesgue measure on $\doo$. Then
$$h_{vol}\geqslant \int \frac{2}{\underline{\alpha}}\ d\mathcal{L},$$
where $\underline{\alpha}$ is defined by
$$\frac{1}{\underline{\alpha}(x)} = \sum_{i=1}^{n} \frac{1}{\underline{\alpha}_i(x)},\ x\in\doo,$$
with $\alpha_1(x) \geqslant \cdots \geqslant \alpha_n(x)$ being the Lyapunov exponents at $x$, counted with multiplicity.
\end{prop}
\begin{proof}
In \cite{bbv}, the authors proved that $h_{vol}$ also measures the exponential growth rate of the volume of spheres:
$$h_{vol} = \limsup_{R\to +\infty} \frac{1}{R} \log vol(S(o,R)),$$
where $S(o,R) = \{x\in\o,\ \d(o,x)=R\}$ is the sphere of radius $R$ about the arbitrary point $o$, and $vol$ denotes the Busemann volume on the sphere. This is well defined because metric balls are convex, hence $S(o,R)$ is $\C^1$ Lebesgue-almost everywhere, so we can consider the Finsler metric induced by $F$ on $S(o,R)$ and define Busemann volume.\\
Fix a probability Lebesgue measure $d\xi$ on the set of directions $H_o\o$ about the point $o$, that we identify with the unit sphere $S(o,1)$. The  volume of the sphere $S(o,R)$ is then given by
$$vol(S(o,R)) = \int f(\xi,R)\ d\xi,$$
where $f(\xi,R) = \det dF(\xi,R)$ with $F$ being the projection about $o$ from $S(o,1)$ to $S(o,R)$. Now, using Jensen inequality and the concavity of $\log$, we get that
$$h_{vol} \geqslant \limsup_{R\to +\infty} \int \frac{1}{R} \log f(\xi,R)\ d\xi;$$
then, the dominated convergence theorem gives
$$h_{vol} \geqslant \int \limsup_{R\to +\infty} \frac{1}{R} \log f(\xi,R)\ d\xi.$$
But it is not difficult to see that, almost everywhere,
$$\limsup_{R\to +\infty} \frac{1}{R} \log f(\xi,R) = \overline{\chi}^+(o,\xi) =  \frac{2}{\underline{\alpha}(\xi^+)},$$
with $\xi^+ = \ph^{+\infty}(o,[\xi])$. Hence the result.
\end{proof}

As a corollary, we can for example state the following result.

\begin{corollary}
Let $(\o,\d)$ be any Hilbert geometry. If the boundary $\doo$ is $\beta$-convex for some $1 \leqslant \beta < +\infty$ then $h_{vol}>0$.
\end{corollary}

\vspace{.6cm}

{\bf Acknowledgements.} I would like to thank my two advisors Patrick Foulon and Gerhard Knieper for all the useful discussions we had in Strasbourg or in Bochum. A great thanks goes to Aur\'elien Bosch\'e who helped me fighting against convex functions. Finally, I thank Fran\c cois Ledrappier for his interest in this work and for encouraging me to write everything down in an article.

\vspace{.6cm}

\bibliographystyle{alpha}
\bibliography{biblio}

\begin{thebibliography}{CVV08}

\bibitem[BBV]{bbv}
G.~Berck, A.~Bernig, and C.~Vernicos.
\newblock Volume entropy of {H}ilbert geometries.
\newblock To appear in Pacific Journal of Mathematics.

\bibitem[Ben03]{benoistqs}
Y.~Benoist.
\newblock Convexes hyperboliques et fonctions quasisym\'etriques.
\newblock {\em Publ. Math. IHES}, 97:181--237, 2003.

\bibitem[Ben04]{benoistcv1}
Y.~Benoist.
\newblock Convexes divisibles 1.
\newblock {\em Algebraic groups and arithmetic, Tata Inst. Fund. Res. Stud.
  Math.}, 17:339--374, 2004.

\bibitem[Ben06]{benoistqi}
Y.~Benoist.
\newblock Convexes hyperboliques et quasiisom\'etries.
\newblock {\em Geom. Dedicata}, 122:109--134, 2006.

\bibitem[Ber09]{bernig}
A.~Bernig.
\newblock {H}ilbert geometry of polytopes.
\newblock {\em Archiv der Mathematik}, 92:314--324, 2009.

\bibitem[CM]{cramponmarquis}
M.~Crampon and L.~Marquis.
\newblock Quotients g\'eom\'etriquement finis des g\'eom\'etries de {H}ilbert.
\newblock In preparation.

\bibitem[CP04]{cv}
B.~Colbois and P.Verovic.
\newblock {H}ilbert geometry for strictly convex domains.
\newblock {\em Geometriae Dedicata}, 105:29--42, 2004.

\bibitem[Cra09]{crampon}
M.~Crampon.
\newblock Entropies of strictly convex projective manifolds.
\newblock {\em Journal of Modern Dynamics}, 3(4):511--547, 2009.

\bibitem[CVV08]{cvv}
B.~Colbois, C.~Vernicos, and P.~Verovic.
\newblock {H}ilbert geometry for convex polygonal domains.
\newblock Preprint, 2008.

\bibitem[dlH93]{delaharpe}
P.~de~la Harpe.
\newblock On {H}ilbert's metric for simplices.
\newblock In {\em Geometric group theory, {V}ol.\ 1}, volume 181 of {\em London
  Math. Soc. Lecture Note Ser.}, pages 97--119. Cambridge Univ. Press, 1993.

\bibitem[FK03]{foertschkarlsson}
T.~Foertsch and A.~Karlsson.
\newblock Hilbert metrics and minkowski norms.
\newblock {\em Journal of Geometry}, 83:22--31, 2003.

\bibitem[Fou86]{foulon86}
P.~Foulon.
\newblock G\'eom\'etrie des \'equations diff\'erentielles du second ordre.
\newblock {\em Ann. Inst. Henri Poincar\'e}, 45:1--28, 1986.

\bibitem[Fou92]{fouloneng}
P.~Foulon.
\newblock Estimation de l'entropie des syst\`emes lagrangiens sans points
  conjugu\'es.
\newblock {\em Ann. Inst. H. Poincar\'e Phys. Th\'eor.}, 57(2):117--146, 1992.
\newblock With an appendix, ``About Finsler geometry'', in English.

\bibitem[Gol90]{goldman}
W.~M. Goldman.
\newblock Convex real projective structures on compact surfaces.
\newblock {\em J. Diff. Geom.}, 31:791--845, 1990.

\bibitem[JM87]{johnsonmillson}
D.~Johnson and J.~J. Millson.
\newblock Deformation spaces associated to compact hyperbolic manifolds.
\newblock In {\em Discrete groups in geometry and analysis, 1984)}, volume~67
  of {\em Progr. Math.}, pages 48--106. Birkh\"auser Boston, 1987.

\bibitem[Kai90]{kaimanovich}
V.~A. Kaimanovich.
\newblock Invariant measures of the geodesic flow and measures at infinity on
  negatively curved manifolds.
\newblock {\em Ann. Inst. Henri Poincar\'e}, 53, n°.4:361--393, 1990.

\bibitem[Kap07]{kapo}
M.~Kapovich.
\newblock Convex projective structures on {G}romov-{T}hurston manifolds.
\newblock {\em Geom. Topol.}, 11:1777--1830, 2007.

\bibitem[Kos68]{koszul}
J.-L. Koszul.
\newblock D\'eformations de connexions localement plates.
\newblock {\em Ann. Inst. Fourier (Grenoble)}, 18(fasc. 1):103--114, 1968.

\bibitem[KV67]{kav}
V.~G. Kac and {\`E}.~B. Vinberg.
\newblock Quasi-homogeneous cones.
\newblock {\em Mat. Zametki}, 1:347--354, 1967.

\bibitem[Ose68]{osedelec}
V.~I. Osedelec.
\newblock A multiplicative ergodic theorem.
\newblock {\em Trans. Moscow Math. Soc.}, 19:197--231, 1968.

\bibitem[Ver08]{vernicos}
C.~Vernicos.
\newblock Lipschitz characterisation of convex polytopal {H}ilbert geometries.
\newblock Preprint, 2008.

\bibitem[Wal08]{walsh}
C.~Walsh.
\newblock The horofunction boundary of the hilbert geometry.
\newblock {\em Advances in Geometry}, 8(4):503--529, 2008.

\end{thebibliography}

\end{document}